\documentclass[onecolumn,draftcls]{IEEEtran}

\usepackage{fourier} 
\usepackage[scaled=0.875]{helvet} 

\usepackage{fullpage}
\usepackage{amsmath,amssymb,amsthm,graphicx,subfig,amsfonts,mathrsfs}
\usepackage{verbatim}
\usepackage{units,color}
\usepackage{setspace}
\usepackage{cite}
\newlength{\minipagewidth}
\usepackage{booktabs}

\def\1{\mathbf{1}}

\newtheorem{definition}{Definition}

\newtheorem{theorem}{Theorem}
\newtheorem{lemma}{Lemma}
\newtheorem{proposition}{Proposition}
\newtheorem{corollary}{Corollary}
\newtheorem{assumption}{Assumption}

\def\bx{\bold{x}}

\def\limsup{\mathop{\rm limsup}}
\def\argmin{\mathop{\rm argmin}}

\def\b{\beta}
\def\e{{\epsilon}}
\def\k{{\kappa}}
\def\R{\mathbb{R}}
\def\re{\R}
\def\rn{\R^n}
\def\bx{{\bold x}}
\def\bw{{\bold w}}
\def\bu{{\bold u}}
\def\bv{{\bold v}}

\def\V{\mathscr{V}}
\def\D{\mathsf{D}}
\def\Pr{\mathbb{P}}
\def\v{\varphi}

\def\dist{{\rm dist}}
\def\V{\mathsf{V}}

\begin{document}

\title{Lyapunov Approach to Consensus Problems}
\author{Angelia Nedi\'c and Ji Liu\thanks{Coordinated Science Laboratory,
University of Illinois, 1308 West Main Street,
Urbana, IL 61801, USA, \texttt{\{angelia,jiliu\}@illinois.edu}.
{Nedi\'c gratefully acknowledges support for this work under
grants NSF CCF 11-11342 and the ONR Navy Basic Research Challenge
N00014-12-1-0998.}}
}
\maketitle
\begin{abstract}
This paper investigates the weighted-averaging dynamic for unconstrained and constrained consensus problems. Through the use of a suitably defined adjoint dynamic, quadratic Lyapunov comparison functions are constructed to analyze the behavior of weighted-averaging dynamic. As a result,
new convergence rate results are obtained that capture the graph structure in a novel way. In particular,
the exponential convergence rate is established for unconstrained consensus with the exponent of the order of
$1-O(1/(m\log_2m))$. Also,
the exponential convergence rate is established for constrained consensus, which extends the existing results limited to the use of doubly stochastic weight matrices.
\end{abstract}

\section{Introduction}
Over the past decade, distributed control has become an active area in control systems society
and there has been considerable interest in distributed computation and
decision making problems of all types. Among these are consensus and flocking problems \cite{Reynolds1987},
distributed averaging \cite{Boyd2006}, multi-agent coverage problems \cite{Bullo2004},
the rendezvous problem \cite{Morse2007}, localization of sensors in a multi-sensor network \cite{Evans2004}
and the distributed management of multi-robot formations \cite{Francis2009}.
These problems have found applications in a wide range of fields including sensor networks,
robotic teams, social networks \cite{Morse2012} and electric power grids \cite{Bullo2013}.
Compared with traditional centralized
control, distributed control is believed more promising for those large-scale complex networks
because of its fault tolerance, cost saving and many inevitable physical constraints such as
limited sensing, computation and communication capabilities.
One of the basic problems arising in decentralized coordination and control
is a consensus problem, also known as an agreement problem \cite{TsThesis,Ts1986,Morse2003,Murray2004,Moreau2005,Ren2005,Basar2007}.
It arises in a number of applications including coordination of UAV's, flocking and formation control,
tracking in network of robots,
and parameter estimation~\cite{Blondel2005,Oh2007,Bullo2009,Mesbahi2010,Martinoli2013, Lopes2008,Sayed2012,RamThesis,alexthesis,kunalthesis}.
In a consensus problem,
we have a set of agents each of which has some initial
variable (a scalar or a vector).  The agents are interconnected over an underlying (possibly time-varying) communication network and each agent has a local view of the network, i.e., each agent is aware of its immediate neighbors in the network and communicates with them only. The goal is to design a distributed
and local algorithm that the agents can execute to agree on a common value asymptotically.
The algorithm needs to be
local in the sense that each agent performs local computations
and communicates only with its immediate neighbors.

In this paper, we present two novel results for consensus problems and averaging dynamics.
The first contribution is the establishment of new convergence rate analysis using Lyapunov approach, which allows us to provide an exponential rate in terms of network structure
(such as longest shortest path) and the properties of the weight matrices. This rate result allows us to establish
that the convergence rate with the ratio of the form
$1-O(1/(m\log_2 m)$ is achievable on special tree-like regular graphs.
The second contribution is the development of the convergence rate result for a constrained consensus,
which is more general than that of~\cite{NOP2010}.
In contrast with~\cite{NOP2010}, we do not require the weight matrices to be doubly stochastic. In fact, it is sufficient to have rooted directed spanning trees contained in the graphs
and the existence of a specific
adjoint dynamic for the linear consensus dynamic. Our analysis makes use of the Lyapunov comparison functions
and absolute probability sequence, which have been developed in~\cite{touri2014} in the more general setting of random graphs (see also~\cite{touribook,tourithesis}).

The paper is organized as follows. In Section~\ref{sec:consensus}, we discuss the weighted-averaging algorithm for consensus problem. In Section~\ref{sec:matrices}, we review some of the recent results for cut-balanced matrices and the related adjoint dynamics for the linear consensus dynamics. Using these results, we
construct suitable Lyapunov comparison functions and study convergence properties of the weighted-averaging algorithm in Section~\ref{sec:wave-rate} for standard consensus problem, while
in Section~\ref{sec:constrained-consensus} we study
a projection-based weighted-averaging algorithm for constrained consensus. We conclude with some remarks in
Section~\ref{sec:concl}.

{\bf Notation}: \  For an integer $m\ge1$, we write $[m]$ to denote the index set $\{1,\ldots,m\}$.
We view vectors as column vectors. We write $x'$ to denote the transpose of a vector $x$ and, similarly, we use
$A'$ for the transpose of a matrix $A$. A vector is stochastic if its entries are nonnegative and
sum to 1. A matrix is said to be stochastic if  its rows are stochastic vectors.
A matrix is doubly stochastic if both
$A$ and its transpose $A'$ are stochastic. A matrix $A$ entries will be denoted by $A_{ij}$ and, also, by
$[A]_{ij}$ when convenient. We use $I$ for the identity matrix.
To differentiate between
the scalar and the vector cases, we use $x_i$ to denote a scalar value associated with agent $i$
and $\bx_i$ for a vector associated with agent $i$. We write $\1$ to denote the vector with all entries equal to 1, where the size of the vector is to be understood from the context.
Given a set $S$ with finitely many elements, we use $|S|$ to denote the cardinality of $S$.
We use $\|\cdot\|$ for the Euclidean norm, while for other $p$-norms we will write $\|\cdot\|_p$.
The Euclidean projection of a point $y$ on a convex closed set $Y$ is denoted by $\Pr_Y[y]$, i.e.,
$\Pr_Y[y]=\argmin_{z\in Y}\|y-z\|$. The distance of a point $y$ to the set $Y$ is denoted by $\dist(y,Y)$, i.e.,
 $\dist(y,Y)=\|y-\Pr_Y[y]\|$.
\section{Unconstrained Consensus}\label{sec:consensus}
We consider a set of $m$ agents, denoted by
$[m]=\{1,\ldots,m\}$. The agents are embedded in a communication network, which is
modeled by a directed graph $G_t=\{[m],E_t\}$, where $E_t\subseteq [m]\times [m]$
is the set of directed links. A link $(i,j)$ indicates
that agent $i$ sends information to agent $j$ at time $t$.
We will work with a sequence $\{G_t\}$ of directed graphs, where each graph $G_t$
contains a directed spanning tree rooted at one of the agents. We refer to such a graph as {\it rooted graph}.
The self-loops will be only virtually added to the graphs to model the fact that every agent has access to
its own state information. We consider the unconstrained consensus problem, formalized as follows.\\
{\bf [Unconstrained Consensus]} \
{\it Design a distributed
algorithm obeying the communication structure given by graph $G_t$ at each time $t$ and ensuring that,
for every set of initial values $\bx_i(0)\in\mathbb{R}^n$, $i\in [m]$, the following limiting behavior
emerges: $\lim_{t\to\infty} \bx_i(t)=c$
for all $i\in [m]$ and some $c\in\R^n$.}\\
The algorithms for solving consensus problems have been mainly constructed using the Laplacians of the graphs
$G_t=([m],E_t)$, e.g.\ see~\cite{Morse2003,Murray2004,Boyd2005},
or weighted-averaging (through the use of stochastic matrices)\cite{Morse2003,Blondel2005,Moreau2005,tourithesis}.
In the scalar case, a well studied approach to the problem is for each agent to
use a linear iterative update rule of the following form
$x(t+1) = W(t)x(t)$ where $x(t)$ is a vector consisting of the $x_i(t)$ and each
$W(t)$ is a stochastic matrix.
One choice is $W(t)= I -\frac{1}{\gamma}L(t)$
where $L(t)$ is the Laplacian of $G_t$ and $\gamma$ is any scalar greater than $m$ (see~\cite{Morse2003}).
An improvement on this choice was obtained in~\cite{Boyd2004,Murray2004} by replacing $\gamma$
with the maximal node degree in the graph $G_t$.
A particularly  interesting improvement, which defines what has come to be known as the Metropolis algorithm,
requires only local information to define the weights $w_{ij}(t)$~\cite{Boyd2005}.
However, most of the Laplacian-based algorithms require that each $W(t)$ is also symmetric which
implicitly require bidirectional communication between agents.
Weighted-averaging algorithms get around this limitation~\cite{TsThesis}.

We will use the weighted-averaging algorithm, which is as follows.
Starting with a vector $\bx_i(0)\in \R^n$, each agent updates at times $t=1,2,\ldots,$ by computing
\begin{eqnarray}\label{eq:wcalgo}
\bx_i(t+1)  = \sum_{j=1}^m A_{ij}(t) \bx_j(t),
\end{eqnarray}
where the weights $A_{ij}(t)$, $i,j\in[m]$, are non-negative
and the positive values satisfy some conditions with respect to the graph $G_t$ structure, to be specified soon.

The dynamic in~\eqref{eq:wcalgo}
is linear,  so we focus on the case where the variables $\bx_i$ are scalars, denoted by $x_i$,
as all the results for the vector case follow immediately by coordinate-wise analysis. The agents' variables
$x_i\in\R$, $i\in[m]$ are stucked to form a vector $x\in\R^m$.
The existing analysis of the weighted-averaging is based on studying the behavior of the left-matrix products.
Specifically, as the iterates $x(t)$ are related over time by the following linear dynamic:
\begin{equation*}
x(t)=A(t)A(t-1)\cdots A(s+1)A(s)x(s)\qquad\hbox{for }t\ge s\ge0,
\end{equation*}
the convergence of the iterates generated by the algorithm
is related to the convergence of the
matrix products $A(t)A(t-1)\cdots A(1)A(0)$, as $t\to\infty$.
In particular, when the matrices
$A(t)A(t-1)\cdots A(1)A(0)$ converge to a rank one matrix, the iterates $x(t)$ converge to a consensus.
Concretely, some conditions on the graphs $G_t$ and the matrices $A(t)$ that yield such a convergence
are given in the following assumption.
\begin{assumption}\label{assume:uniform}
Let $\{G_t\}$ be a graph sequence and $\{A(t)\}$ be a sequence of $m\times m$ matrices
that satisfy the following conditions:
\begin{itemize}
\item[(a)]
Each $A(t)$ is a stochastic matrix that is compliant with the graph $G_t$, i.e.,
$A_{ij}(t)>0$ when $(j,i)\in E_t$, for all $t$.
\item[(b)] (Aperiodicity) The diagonal entries of each $A(t)$ are positive, $A_{ii}(t)>0$ for all $t$ and $i\in[m]$.
\item[(c)] (Uniform Positivity) There is a scalar $\beta>0$ such that $A_{ij}(t)\ge\beta$ whenever $A_{ij}(t)>0$.
\item[(d)] (Irreducibility) Each $G_t$ is strongly connected.
\end{itemize}
\end{assumption}

The convergence properties of the weighted-averaging algorithm have been extensively studied under
Assumption~\ref{assume:uniform} (see~\cite{TsThesis,Morse2003,Blondel2005,Morse2008a}).
Actually, in this case the matrix sequence $\{A(t)\}$ is known to be {\it ergodic} in the sense that
the limit
\[\lim_{t\to\infty} A(t)\cdots A(k+1)A(k)\quad \hbox{exists for all }k\ge0.\]
Moreover, it is known that the convergence rate of these products is geometric.
The convergence rate question has been studied in~\cite{Morse2008b,Nedic2009a,Morse2011,Nedic2009b,Ts2013}
for deterministic matrix sequences and in~\cite{Fagnani2008,Bajovic2013,touri2014} for random sequences.
In~\cite{Nedic2009b,Ts2009,Morse2011p}, the convergence rate question was addressed for the cases when
the matrices $A(t)$ are doubly stochastic;
the best polynomial-time bound on the convergence time was given in~\cite{Nedic2009b}.
Specifically, the following result is well known.
\begin{theorem}\label{thm:known}
{[Lemma 5.2.1 in~\cite{TsThesis}, Lemma 5 in~\cite{Nedic2009b}] }
Under Assumption~\ref{assume:uniform} we have
\[\lim_{t\to\infty} A(t)\cdots A(k+1)A(k)=\1\phi'(k)\qquad\hbox{for all }k\ge0,\]
where each $\phi(k)$ is stochastic vector. Furthermore, the convergence rate is geometric:
for all $t\ge k\ge0$,
\[\|A(t)\cdots A(k+1)A(k)-\1\phi'(k)\|^2\le C q^{t-k},\] 
where the constants $C>0$ and $q\in (0,1)$ depend only on $m$ and $\beta$.
When the matrices $A(t)$ are doubly stochastic, we have for all $t\ge k\ge0$,
\[\left\|A(t)\cdots A(k+1)A(k)-\frac{1}{m}\1\1' \right\|^2\le  \left(1-\frac{\beta}{2m^2}\right)^{t-k}.\]
\end{theorem}

These and the other existing rate results are not explicitly capturing the structure of the graph $G_t$ such as the longest shortest path for example.
In what follows, we develop such rate results by adopting dynamic system point of view
and applying Lyapunov approach.
This approach allows us to characterize the convergence of the weighted-averaging algorithm
with a more explicit dependence on the graph structure than that of Theorem~\ref{thm:known}.
In particular, we work with a quadratic Lyapunov comparison function proposed by
Touri~\cite{touri2011}, and
we build on the results developed in Touri's thesis~\cite{tourithesis}
(see also~\cite{touribook,touri2014}). In this approach,
an absolute probability sequence of matrices $A(t)$ play a critical role in the construction of a Lyapunov comparison function
and in establishing its rate of decrease along the iterates of the algorithm.
\section{Absolute Probability Sequence}\label{sec:matrices}
We embark on a study of the important features of stochastic matrices for convergence of the weighted-averaging
method.
The development here makes use of the notion of an absolute probability sequence
associated with a sequence $\{A(t)\}$ of stochastic matrices. This notion
was introduced by Kolmogorov~\cite{Kolmogorov}.

\begin{definition}\label{def:ap}
{\rm \cite{Kolmogorov}}
Let $\{A(t)\}$ be a sequence of stochastic matrices. A sequence of stochastic vectors $\{\pi(t)\}$ is
an absolute probability sequence for $\{A(t)\}$ if 
\begin{equation}\label{eq:adjoint}
\pi'(t)=\pi'(t+1) A(t)\qquad\hbox{for all }t\ge0.\end{equation}
\end{definition}

Blackwell~\cite{Blackwell1945} has shown that {\it every sequence
of stochastic matrices has an absolute probability sequence}.
As a direct consequence of Blackwell's result, 
 every {\it ergodic} sequence of stochastic matrices has an absolute probability sequence
 (an earlier result due to Kolmogorov~\cite{Kolmogorov}).
 In particular, for an ergodic sequence $\{A(t)\}$ of stochastic matrices we have
  \begin{equation}\label{eq:phi}
\lim_{\tau\to\infty} A(\tau)A(\tau-1)\cdots A(t+1) A(t) =\1\phi'(t),
\end{equation}
and $\{\phi(t)\}$ is an absolute probability sequence for $\{A(t)\}$.
In general, a sequence $\{A(t)\}$ of stochastic matrices may have more than one absolute probability sequence.
The following example has been communicated to us by B.~Touri:
if each of the matrices $A(t)$ is invertible and each $A(t)^{-1}$ is stochastic, then for any stochastic vector
$u$, we can construct an absolute probability sequence for $\{A(t)\}$ by letting
$\pi'(0)=u'$ and $\pi'(t+1)=\pi'(t)A(t)^{-1}$ for all $t\ge0$.
Thus, $\{A(t)\}$ has infinitely many absolute probability sequences.

We show that the absolute probability sequence is unique for an ergodic stochastic matrix sequence.

\begin{lemma}\label{lem:unique}
Let $\{A(t)\}$ be an ergodic sequence of stochastic matrices (cf.~\eqref{eq:phi}).
Then, the vector sequence $\{\phi(t)\}$ is the unique absolute probability sequence for $\{A(t)\}$.
\end{lemma}
\begin{proof}
Assume that
$\{\pi(t)\}$ is another absolute probability sequence for $\{A(t)\}$. Then, we have
\[\pi'(t)
=\pi'(t+\tau)A(t+\tau-1)\cdots A(t+1)A(t)\]
for all $\tau\ge 1$ and $t\ge0$.
Thus,
\begin{align*}
\pi'(t)
&=\pi'(t+\tau)\left(A(t+\tau-1)\cdots A(t)-\1\phi'(t)\right) \cr
& \ +\pi'(t+\tau)\1\phi'(t)\cr
& = \pi'(t+\tau)\left(A(t+\tau-1)\cdots A(t)-\1\phi'(t)\right) + \phi'(t),
\end{align*}
where in the second equality we use
$\pi'(t+\tau)\1=1$. By letting $\tau\to\infty$ and using $\|\pi'(s)\|_1=1$, we obtain
\begin{align*}
&\|\pi'(t)-\phi'(t)\|_1\cr
& \le  \limsup_{\tau\to\infty}\left(\|\pi'(t+\tau)\|_1\|A(t+\tau-1)\cdots A(t)-\1\phi'(t)\|_\infty\right) \cr
&\le  \lim_{\tau\to\infty}\|A(t+\tau-1)\cdots A(t)-\1\phi'(t)\|_\infty
 =  0.\end{align*}
\end{proof}

In the subsequent development, it will be important that a sequence $\{A(t)\}$ of
stochastic matrices has an absolute probability sequence of vectors $\pi(t)$
whose entries are uniformly bounded away from zero.
This is the case when each matrix $A(t)$ is doubly stochastic,
as we can use $\pi'(t)=\frac{1}{m}\1$.
Another class of matrices that have this property
is a subclass of {\it cut-balanced} matrices~\cite{touri2014} (see there the class $\mathscr{P}^*$).
(See  Hendrickx and Tsitsiklis~\cite{julien2013} for cut-balancedness as studied for continuous-time systems,
and Touri~\cite{touribook,touri2014}
and Bolouki and Malham\'e~\cite{bolouki2012} for discrete-time systems.)

In what follows, we will work under the following assumption, where we view a rooted tree
$\mathsf{T}_t$ as a collection of directed edges from $E_t$.
\begin{assumption}\label{assume:minimal}
Let $\{G_t\}$ be a graph sequence and $\{A(t)\}$ be a matrix sequence such that:
\begin{itemize}
\item[(a)]
 (Partial Irreducibility) Each graph $G_t$ is rooted and
each $A(t)$ is a stochastic matrix that is compliant with a rooted directed spanning tree $\mathsf{T}_t$ of  $G_t$, i.e., $ A_{ij}(t)>0$ whenever $(j,i)\in \mathsf{T}_t$ for all $t\ge0$.
\item[(b)]
 (Aperiodicity) The diagonal entries of each $A(t)$ are positive, $A_{ii}(t)>0$ for all $t$, and $i\in[m]$.
 \item[(c)]
(Partial Uniform Positivity)
There is a scalar $\beta>0$ such that
$A_{ii}(t)\ge \beta$ and $A_{ij}(t)\ge\beta$ for all $(j,i)\in \mathsf{T}_t$ and for all $t\ge0$.
\item[(d)] The matrix sequence $\{A(t)\}$ has an absolute probability sequence
$\{\pi(t)\}$ that is uniformly bounded away from zero, i.e.,
there is $\delta\in(0,1)$ such that $\pi_i(t)\ge\delta$ for all $i$ and $t$.
\end{itemize}
\end{assumption}
One can show that Assumption~\ref{assume:uniform} implies Assumption~\ref{assume:minimal}.

\section{Weighted-Averaging Algorithm}\label{sec:wave-rate}

We analyze convergence properties of the weighted-averaging algorithm in~\eqref{eq:wcalgo}
by using a suitable Lyapunov comparison function.
\subsection{Lyapunov Comparison Function}\label{sec:lyapunov}
As indicated in~\cite{touri2014}, there are many possible constructions of Lyapunov comparison functions
by using convex functions and absolute probability sequences, i.e., the adjoint dynamic in~\eqref{eq:adjoint}.
Here, we focus on the quadratic case, where the function is of the form:
\begin{equation}\label{eq:vq}
\v(x,\nu)\triangleq\sum_{i=1}^m \nu_i x_i^2 -  (\nu'x)^2
\ \hbox{for $x\in\R^m$ and $\nu\in\R^m_+$},
\end{equation}
for suitably chosen vectors $\nu$ (which will vary with time).
The function $\v$ has an equivalent form:
\begin{equation}\label{eq:phi-2}
\v(x,\nu)=\sum_{i=1}^m \nu_i\left( x_i -  (\nu'x)\right)^2
\hbox{for $x\in\R^m$ and $\nu\in\R^m_+$},
\end{equation}
which can be seen
by expanding $\left( x_i -  (\nu'x)\right)^2$.
The quadratic function $s\mapsto s^2$ has exact second order expansion, which allows us
to obtain the exact expression for the difference
$\v(Ax,\nu)-\v(x,A'\nu)$ for a stochastic matrix $A$, as seen in the following lemma.

\begin{lemma}\label{lem:qlyap}
Let $A$ be an $m\times m$ stochastic matrix. We then have
 for all $x\in\R^m$ and all $\nu\in\R^m_+$,
\[\v(Ax,\nu) = \v(x,A'\nu)
- \frac{1}{2}\sum_{i=1}^m \nu_i\sum_{j=1}^m\sum_{\ell=1}^m A_{ij}A_{i\ell}(x_j-x_\ell)^2.\]
\end{lemma}
\begin{proof}
By the definition of $\v$ we have
$\v(Ax,\nu) = \sum_{i=1}^m \nu_i ([Ax]_i))^2 - (\nu'Ax)^2,$
where $[Ax]_i=\sum_{j=1}^m A_{ij} x_j$.
We fix an arbitrary index $i$, and we expand $([Ax]_i)^2$ to obtain
\[([Ax]_i)^2=\sum_{j=1}^m\sum_{\ell=1}^m A_{ij}A_{i\ell}x_jx_\ell.\]
Since $x_jx_\ell=\frac{1}{2}\left( x_j^2 + x_\ell^2 - (x_j-x_\ell)^2 \right),$
it follows that
\begin{eqnarray*}
([Ax]_i)^2
&=&\frac{1}{2}\sum_{j=1}^mA_{ij}\left(\sum_{\ell=1}^m A_{i\ell}\right) x_j^2
+\frac{1}{2}\sum_{\ell=1}^m A_{i\ell}\left( \sum_{j=1}^m A_{ij}\right) x_\ell^2 \cr
&& - \frac{1}{2}\sum_{j=1}^m\sum_{\ell=1}^m A_{ij}A_{i\ell}(x_j-x_\ell)^2.
\end{eqnarray*}
Note that $\sum_{\ell=1}^m A_{i\ell}=1$ since the matrix $A$ is stochastic, thus implying
\begin{eqnarray*}
([Ax]_i)^2
&=& \sum_{j=1}^mA_{ij}x_j^2 -  \frac{1}{2}\sum_{j=1}^m\sum_{\ell=1}^m A_{ij}A_{i\ell}(x_j-x_\ell)^2.
\end{eqnarray*}
By multiplying the preceding relation with $\nu_i$ and by summing over $i$, we obtain
\begin{align*}
\v(Ax,\nu)
= & \sum_{j=1}^m\left(\sum_{i=1}^m \nu_i  A_{ij}\right) x_j^2 \cr
& \ -
\frac{1}{2}\sum_{i=1}^m \nu_i\sum_{j=1}^m\sum_{\ell=1}^m A_{ij}A_{i\ell}(x_j-x_\ell)^2 - (\nu'Ax)^2.
\end{align*}
Observe that $\sum_{i=1}^m \nu_i  A_{ij}=[A'\nu]_j$. Therefore, by using the definition
of the function $\v$ we find
\begin{eqnarray*}
\v(Ax,\nu)
= \v(x,A'\nu)
- \frac{1}{2}\sum_{i=1}^m \nu_i\sum_{j=1}^m\sum_{\ell=1}^m A_{ij}A_{i\ell}(x_j-x_\ell)^2.
\end{eqnarray*}
\end{proof}
Lemma~\ref{lem:qlyap} provides one of the fundamental relations in the assessment of the convergence rate
of the weighted-averaging algorithm.
\subsection{Convergence Rate Analysis}\label{sec:conv}
In this part, we will first
show the convergence of the weighted-averaging algorithm~\eqref{eq:wcalgo}
for the scalar case, by considering the decrease of $\v(x(t),\pi(t))$ over time
along the iterate sequence $\{x(t)\}$, where $\{\pi(t)\}$ is an absolute probability sequence of $\{A(t)\}$.
The decrease of this function in time can be captured exactly, as follows.
Since $x(t+1)=A(t) x(t)$ and the matrices $A(t)$ are stochastic,
by Lemma~\ref{lem:qlyap} it follows
\begin{eqnarray*}
\v\left( x(t+1),\pi(t+1)\right) & = & \v\left( A(t)x(t),\pi(t+1) \right)  \cr
& = & \v\left( x(t),A'(t)\pi(t+1) \right) - D(t),\end{eqnarray*}
where
\begin{align}\label{eq:dec}
D(t) =\frac{1}{2}\sum_{i=1}^m \pi_i(t+1)
\sum_{j=1}^m\sum_{\ell=1}^m A_{ij}(t)A_{i\ell}(t)\left(x_j(t)-x_\ell(t)\right)^2.
\end{align}
By the definition of the adjoint dynamics in~\eqref{eq:adjoint}, we have
$A'(t)\pi(t+1)=\pi(t)$, implying that
\begin{align}\label{eq:lyap-dec}
\v\left( x(t+1),\pi(t+1) \right) = \v\left( x(t),\pi(t) \right)- D(t).\end{align}

Note that function $\v(\cdot,\nu)$ induces a semi norm on
$\R^m$ when $\nu$ is a stochastic vector, and it induces a norm when all the entries $\nu_i$ are positive.
Thus, to properly bound the decrease $D(t)$ (cf.~\eqref{eq:dec}) of the function $\v\left(x(t),\pi(t)\right)$, one would like to have
$\phi_i(t)>\delta$ for all $i$, for some $\delta$ and for all sufficiently large $t$.
This property can be ensured (for all $t$) by
requiring the additional properties on the matrix sequence $\{A(t)\}$ and the graph sequence $\{G_t\}$ such as
cut-balancedness
(see Lemma 9 in~\cite{touri2014}).
Once all $\pi_i(t)$ are bounded uniformly away from zero,
to further bound $D(t)$ from below, we would also like
that the value of the sum $\sum_{i=1}^m\sum_{j=1}^m\sum_{\ell=1}^m A_{ij}(t)A_{i\ell}(t)$ does not
vanish in time. These properties are ensured by Assumption~\ref{assume:minimal}, which we use
to establish the key relation for the decrease amount $D(t)$,  as seen in the following lemma.

\begin{lemma}\label{lem:Dbound}
Let Assumption~\ref{assume:minimal} hold.
Consider the decrement $D(t)$ given by: for $t\ge0,$
\[D(t) =\frac{1}{2}\sum_{i=1}^m \pi_i(t+1)\sum_{j=1}^m\sum_{\ell=1}^m A_{ij}(t)A_{i\ell}(t)\left(x_j(t)-x_\ell(t) \right)^2.\]
Then, the decrement is bounded from below as follows:
\begin{eqnarray*}
D(t)\ge \frac{\delta \b^2}{4 p^*(t) }\,\max_{j,\ell\in [m]} \left(x_{j}(t)-x_{\ell}(t) \right)^2
\quad\hbox{for }t\ge0,
\end{eqnarray*}
where $\beta>0$ and $\delta>0$ are from
Assumptions~\ref{assume:minimal}(c) and~\ref{assume:minimal}(d), respectively,
while $p^*(t)$ is the maximum number of links in any of the directed paths in the tree
$\mathsf{T}_t$ of Assumption~\ref{assume:minimal}(a).
\end{lemma}
\begin{proof}
We let $t\ge0$ be arbitrary but fixed.
By Assumption~\ref{assume:minimal}(d), it follows that
\[D(t)\ge \frac{\delta }{2}\sum_{i=1}^m\sum_{j=1}^m\sum_{\ell=1}^m A_{ij}(t)A_{i\ell}(t)\left(x_j(t)-x_\ell(t) \right)^2.\]
Let us observe that
\begin{eqnarray*}
\sum_{i=1}^m\sum_{j=1}^m\sum_{\ell=1}^m A_{ij}(t)A_{i\ell}(t)
& =& \sum_{j=1}^m\sum_{\ell=1}^m \left(A_{:j}(t)\right)'A_{:\ell}(t),
\end{eqnarray*}
where $A_{:j}$ denotes $j$th column vector of a matrix $A$.
From this relation, we further obtain
\begin{eqnarray}\label{eq:relD}
D(t)
\ge\delta \sum_{j=1}^m\sum_{\ell=j+1}^m \left(A_{:j}(t)\right)' A_{:\ell}(t)\left(x_j(t)-x_\ell(t) \right)^2.
\end{eqnarray}
Let $j^*$ and $\ell^*$ be two agents such that
\begin{equation}\label{eq:maxinv}
\max_{j,\ell\in [m] }|x_j(t)-x_\ell(t)|=|x_{j^*}(t)-x_{\ell^*}(t)|.\end{equation}
Note that for any node $v$ we must have
\begin{equation}\label{eq:dist-two}
\max\{ |x_{v}(t)-x_{j^*}(t)|, |x_{v}(t)-x_{\ell^*}(t)|\}\ge \frac{1}{2}|x_{j^*}(t)-x_{\ell^*}(t)|,
\end{equation}
for otherwise by the triangle inequality for the norm we would have
\begin{eqnarray*}
|x_{j^*}(t)-x_{\ell^*}(t)| & \le & |x_{v}(t)-x_{j^*}(t)| +|x_{v}(t)-x_{\ell^*}(t)| \cr
& < & |x_{j^*}(t)-x_{\ell^*}(t)|, \end{eqnarray*}
which is a contradiction.

According to Assumption~\ref{assume:minimal}(a), in the graph $G_t$ there is
a rooted directed spanning tree $\mathsf{T}_t$. Let agent $v^*$ be the root node of this tree.
Then, relation~\eqref{eq:dist-two} holds for $v=v^*$.
Without loss of generality let us assume that $j^*$ attains the maximum in~\eqref{eq:dist-two} when $v=v^*$,
i.e., $|x_{v^*}(t)-x_{j^*}(t)|\ge |x_{v^*}(t)-x_{\ell^*}(t)|,$
so that we have
\begin{equation}\label{eq:keyrelvj}
|x_{v^*}(t)-x_{j^*}(t)| \ge \frac{1}{2}|x_{j^*}(t)-x_{\ell^*}(t)|.
\end{equation}
Since $v^*$ is the root of the directed spanning tree $\mathsf{T}_t$,
there must exist a path from $v^*$ to $j^*$, i.e.,
$v^*=j_0\to j_1\to j_2\to\cdots\to j_{p}= j^*$ with links $(j_{\kappa},j_{\kappa+1})$ in the tree $\mathsf{T}_t$.
Then, using~\eqref{eq:relD} we can write
\begin{eqnarray}\label{eq:relD1}
D(t)\ge \delta \sum_{\kappa=0}^{p-1} \left(A_{:j_{\kappa}}(t)\right)' A_{:j_{\kappa+1}}(t)
\left(x_{j_\kappa}(t)-x_{j_{\kappa+1}}(t) \right)^2.
\end{eqnarray}
We now look at the coefficients $\left(A_{:j_\k}(t)\right)'A_{: j_{\k+1}}(t)$
in~\eqref{eq:relD1} along the path $v^*=j_0\to j_1\to j_2\to\cdots\to j_{p}= j^*$
For each $\k=0,\dots, p-1$, we have
\begin{align}\label{eq:c1}
\left(A_{:j_{\k}}(t)\right)'A_{:j_{\k+1}}(t) & =  \sum_{i=1}^m A_{ij_{\k}}(t) A_{i j_{\k+1}}(t) \cr
&\ge A_{j_{\k+1}j_{\k} }(t) A_{j_{\k+1}  j_{\k+1}}(t) \ge \beta^2,
\end{align}
where the last inequality follows by Assumption~\ref{assume:minimal}(c).
From relations~\eqref{eq:relD1} and~\eqref{eq:c1} we see that
\begin{eqnarray}\label{eq:relD2}
D(t)\ge \delta \b^2\sum_{\k=0}^{p-1}\left(x_{j_{\k}}(t)-x_{j_{\k}+1}(t) \right)^2.
\end{eqnarray}
Since the function $s\to s^2$ is convex, we have
\begin{eqnarray*}
\frac{1}{p}  \sum_{\k=0}^{p-1}\left( x_{j_\k}(t) - x_{j_{\k+1}}(t) \right)^2
& \ge & \left( \frac{1}{p}\sum_{\k=0}^{p-1}\left( x_{j_\k}(t) - x_{j_{\k+1}}(t)\right) \right)^2\cr
& = & \left( \frac{1}{p} \left(x_{j_0}(t)-x_{j_p}(t)\right) \right)^2,
\end{eqnarray*}
implying that
\[\sum_{\k=0}^{p-1}\left( x_{j_\k}(t) - x_{j_{\k+1}}(t) \right)^2\ge \frac{1}{p} \left(x_{j_0}(t) - x_{j_p}(t)\right)^2.\]
Therefore, from the preceding relation and~\eqref{eq:relD2}, by recalling that $j_0=v^*$ and $j_p=j^*$,
we obtain
\begin{eqnarray}\label{eq:relD3}
D(t)\ge \frac{\delta \b^2}{p}\left(x_{v^*}(t)-x_{j^*}(t) \right)^2.
\end{eqnarray}

Finally, using inequality~\eqref{eq:keyrelvj} in relation~\eqref{eq:relD3} we obtain
\begin{eqnarray}\label{eq:relD4}
D(t)\ge \frac{\delta \b^2}{4p}\left(x_{j^*}(t)-x_{\ell^*}(t) \right)^2.
\end{eqnarray}
Recall that $p$ is the number of links in the path from $v^*$ to $j^*$ in
the directed spanning tree $\mathsf{T}_t$ (rooted at $v^*$) of the graph $G_t$. Thus, $p$ is bounded from above
by the maximal number of links along the path from $v$ to any other node in the graph $G_t$,
where the paths are taken along the directed spanning tree rooted at $v^*$.
We note that $p^*$ depends on time $t$ which was fixed so far, and we have suppressed this dependence on $t$.
Recall, further that $j^*$ and $\ell^*$ are agents
with the maximal difference $|x_j(t)-x_\ell(t)|$ (see Eq.~\eqref{eq:maxinv}).
Thus, from the relation in ~\eqref{eq:relD4} we have
$D(t)\ge \frac{\delta \b^2}{4p^*(t) }\max_{j,\ell\in [m]} \left(x_{j}(t)-x_{\ell}(t) \right)^2.$
\end{proof}

Before stating our main result, we provide an auxiliary lemma for use in the forthcoming
analysis.

\begin{lemma}\label{lem:max}
For any stochastic vector $\nu\in\R^m$ and any $x\in\R^m$ it holds that
\[\sum_{i=1}^m \nu_i(x_i-\nu'x)^2\le \max_{1\le j,\ell\le m}(x_j-x_\ell)^2.\]
\end{lemma}
\begin{proof}
Since $\nu$ is stochastic vector, it follows that
$\sum_{i=1}^m \nu_i(x_i-\nu'x)^2 
\le \max_{1\le \k\le m} (x_\k-\nu'x)^2.$
Without loss of generality, let us assume that the preceding maximum is attained
for $\k=1$,
\[(x_1 - \nu'x)^2=\max_{1\le \k\le m} (x_\k-\nu'x)^2,\]
and note that, since $\nu'\1=1$ we can write
$x_1 - \nu'x = x_1\nu'\1  - \nu'x =\nu' (x_1\1 -  x).$
Using the preceding relation, the fact that $\nu$ is a stochastic vector, and
the convexity of the function $s\mapsto s^2$, we obtain
\begin{eqnarray*}
(x_1 - \nu'x)^2 &= &\left(\nu' (x_1\1 -  x)\right)^2 \le \sum_{i=1}^m\nu_i (x_1-x_i)^2\cr
&\le & \max_{1\le \ell\le m}(x_1-x_\ell)^2.\end{eqnarray*}
Therefore, we have
\[\sum_{i=1}^m \nu_i(x_i-\nu'x)^2\le \max_{1\le \ell\le m}(x_1-x_\ell)^2
\le \max_{1\le j,\ell\le m}(x_j-x_\ell)^2.\]
\end{proof}

With Lemma~\ref{lem:Dbound} in place, we can now establish a key relation for
the quadratic comparison function. The convergence result of the
weighted-averaging algorithm, as well as its convergence rate estimates, will follow from this relation.

\begin{theorem}\label{thm:key}
Under Assumption~\ref{assume:minimal}, for the iterates $\{x(t)\}$ generated by
the weighted-averaging algorithm~\eqref{eq:wcalgo} with any initial vector $x(0)\in\R^m$,
we have for any $t\ge k\ge0$,
\begin{eqnarray*}
&& \sum_{i=1}^m \pi_i(t)\left(x_{i}(t) - \pi(0)'x(0) \right)^2 \cr
 &&\le  \left(1-\frac{\delta \b^2}{4p^*}\right)^{t-k}
 \sum_{j=1}^m \pi_j(k)\left(x_{j}(k) - \pi(0)'x(0) \right)^2,
\end{eqnarray*}
where $\beta>0$ and $\delta>0$ are from
Assumptions~\ref{assume:minimal}(c) and~\ref{assume:minimal}(d), while
$p^*=\max_{s\ge0} p^*(s)$ where
$p^*(s)$ is the longest shortest path in the tree $\mathsf{T}_s$
of Assumption~\ref{assume:minimal}(a).
\end{theorem}
\begin{proof}
The stated relation for $t=k$ can be seen to hold by inspection. Consider now $t>k\ge0$ where $t$ and $k$ are arbitrary but fixed.
From relations~\eqref{eq:dec}--\eqref{eq:lyap-dec} and
 Lemma~\ref{lem:Dbound} we obtain for all $t\ge0$,
\begin{eqnarray*}
\v(x(t+1),\pi(t+1))
\le \v(x(t),\pi(t))-\frac{\delta \b^2}{4p^*(t) }\max_{j,\ell\in [m]} \left(x_{j}(t)-x_{\ell}(t) \right)^2.
\end{eqnarray*}
From Lemma~\ref{lem:max} it follows that
\[\max_{1\le j,\ell\le m}(x_j(t)-x_\ell(t))^2\ge \sum_{j=1}^m \pi_j(t)\left( x_j(t) - \pi(t)'x(t) \right)^2,\]
thus implying that for all $t\ge0$,
\begin{equation*}
\v(x(t+1),\pi(t+1))
 \le  \left(1-\frac{\delta \b^2}{4p^*(t) }\right)\sum_{j=1}^m \pi_j(t)\left(x_{j}(t) - \pi(t)'x(t) \right)^2.
\end{equation*}
Hence, for all $t\ge0$,
\begin{eqnarray*}
&& \sum_{i=1}^m \pi_i(t+1)\left(x_{i}(t+1) - \pi(t+1)'x(t+1) \right)^2 \cr
 && \le  \left(1-\frac{\delta \b^2}{4p^*(t) }\right)\sum_{j=1}^m \pi_j(t)\left(x_{j}(t) - \pi(t)'x(t) \right)^2.
\end{eqnarray*}
Furthermore, from the dynamics in~\eqref{eq:wcalgo} and~\eqref{eq:adjoint}
we can see that for all $t\ge1$,
\begin{align*}
\pi(t)'x(t) &=\pi(t)'A(t-1) x(t-1)=\pi(t-1)'x(t-1) \cr
& =\cdots=\pi(0)'x(0),
\end{align*}
which yields for all $t\ge0$,
\begin{eqnarray*}
&& \sum_{i=1}^m \pi_i(t+1)\left(x_{i}(t+1) - \pi(0)'x(0) \right)^2 \cr
 && \le  \left(1-\frac{\delta \b^2}{4p^*(t) }\right)\sum_{j=1}^m \pi_j(t)\left(x_{j}(t) - \pi(0)'x(0) \right)^2.
\end{eqnarray*}
The stated relation follows by recursively using the preceding inequality for $t, t-1, \ldots,k$,
and then using $p^*(s)\le p^*$ for all $s$.
\end{proof}

Theorem~\ref{thm:key} captures the convergence rate in terms of
the longest shortest paths in the graph sequence. The quotient $q=1-\frac{\delta \b^2}{4p^*} $
indicates the rate at which the information is diffused in the graphs $\{G_t\}$ over time,
with a small $q$ being desirable for  a fast diffusion.

Several immediate consequences of Theorem~\ref{thm:key} are in place.
First, we observe that from Theorem~\ref{thm:key} it follows that
the agent iterates converge to the consensus value $\pi(0)'x(0)$,
by virtue of the lower boundedness property of
the absolute probability sequence (Assumption~\ref{assume:minimal}(d)),
i.e., $\lim_{t\to\infty}x_i(t)= \pi(0)'x(0)$ for all $i\in[m].$
When the agent variables $\bx_i$ are vectors,
then by applying Theorem~\ref{thm:key} to each coordinate
of the vectors, we can see that the iterates $\bx_i(t)$ generated by the weighted-averaging algorithm
are such that for any initial vectors $\bx_i(0)\in\R^n$, $i\in[m]$,
for each coordinate index $\ell\in [n]$, and
for all $t\ge k\ge 0$, we have
\begin{eqnarray*}
&& \sum_{i=1}^m \pi_i(t)\left( [\bx_{i}(t)]_\ell - c_\ell \right)^2\cr
&&\le \left(1-\frac{\delta \b^2}{4p^*} \right)^{t-k}
\sum_{j=1}^m \pi_j(k)\left([\bx_{j}(k)]_\ell - c_\ell\right)^2,
\end{eqnarray*}
where $c_\ell=\sum_{i=1}^m \pi_i(0)' [\bx_i(0)]_\ell$ for all $\ell\in [n].$
By summing these relations over all coordinate indices $\ell\in[n]$, we obtain the following result.
\begin{corollary}\label{cor:vector}
Consider the vector-valued consensus problem and let Assumption~\ref{assume:minimal} hold.
Then,  the iterates $\{\bx_i(t)\}$, $i\in[m]$ generated by the weighted-averaging algorithm are such that
for any initial vectors $\bx_i(0)\in\R^n$,
\begin{eqnarray*}
\sum_{i=1}^m \pi_i(t)\left\|\bx_{i}(t) - c\right\|^2
\le \left(1-\frac{\delta \b^2}{4p^*} \right)^{t-k}
\sum_{j=1}^m \pi_j(k)\left\|\bx_{j}(k) - c\right\|^2 
\end{eqnarray*}
for all $t\ge k\ge 0$, where the vector $c\in\R^n$ has coordinates given by
$c_\ell=\sum_{i=1}^m \pi_i(0)' [\bx_i(0)]_\ell$ for all $\ell\in [n].$
\end{corollary}

Some further implications of Theorem~\ref{thm:key} are discussed in the following section.

\subsection{Implications of Theorem~\ref{thm:key} }
We present some implications of Theorem~\ref{thm:key} regarding the improvement of the best known rate of $O(m^2)$ and the convergence properties of the matrix products $A(t)\cdots A(k+1)A(k)$.

Let Assumption~\ref{assume:minimal} hold, and assume also that the weight matrices $A(t)$, $t\ge0$, are doubly stochastic. Then, we have $\pi(t)=\frac{1}{m}\1$ and the relation
of Theorem~\ref{thm:key} reduces to (after multiplication by $m$):
\begin{equation}\label{eq:dbs}
\left\|x(t) - \bar x(0)\1 \right\|^2
 \le \left(1-\frac{\b^2}{4mp^*}\right)^{t-k}
 \left\|x(k) - \bar x(0)\1 \right\|^2,
\end{equation}
with $\bar x(0)=\frac{\1'x(0)}{m}$.
Since the maximum path length from the root to any other node cannot exceed $m-1$, i.e., $p^*(s)\le m-1$,
it follows that
\begin{equation*}
\left\|x(t) - \bar x(0)\1\right\|^2
 \le \left(1-\frac{\b^2}{4m(m-1) }\right)^{t-k}
 \left\|x(k) - \bar x(0)\1\right\|^2.
\end{equation*}
Thus, when $\beta$ does not depend on $m$,
the convergence rate has dependency of $O(m^2)$ in terms of the number $m$ of agents,
which is the same as the rate result in~\cite{Nedic2009b}]; see Theorem~\ref{thm:known}.

Suppose now that we want to construct the graphs $G_t$ such that Assumption~\ref{assume:minimal} holds
and we want to get the most favorable rate dependency on $m$. In this case, the following result is valid.
\begin{theorem}\label{thm:mlogm}
There is a sequence $\{G_t\}$ of
regular undirected graphs  such that for all $x(0)\in\R^m$ and all $t\ge k\ge0$,
\begin{equation*}
\left\|x(t) - \bar x(0)\1\right\|^2
 \le q^{t-k}
 \left\|x_{j}(k) - \bar x(0)\1\right\|^2,
\end{equation*}
with $q=1-\frac{1}{4^3m \lceil\frac{\log_2 m}{2}\rceil }$ and $\bar x(0)=\frac{\1'x(0)}{m}$.
\end{theorem}
\begin{proof}
We will construct an undirected graph sequence $\{G_t\}$ that satisfies
Assumption~\ref{assume:minimal}. Let $m=2^d$ for some integer $d\ge1$.
Let $t$ be arbitrary but fixed time. Select $2^d-1$ agents and construct an undirected binary tree
with these agents as nodes. Next, add one extra agent as a root with a single child (see Figure~\ref{fig:regular1}).
Thus, each agent $i$ except for the root and the leaf agents has the degree equal to 3. Consider, now connecting
all leaf-nodes with undirected edges (see Figure~\ref{fig:regular2}).
Now, all leaf-agents have degree equal to 3 except for the far most left and far most right
agents, each of which has the degree equal to~2. Connect these two agents to the root node (see Figure~\ref{fig:regular3}).
In this way, the far most left and far most right leaf agents, as well as the root agent have degree 3.
\begin{figure}[h!] \vskip -1.5pc
\centering
\subfloat[ Binary tree]{\label{fig:regular1}
\includegraphics[width=0.3\linewidth] {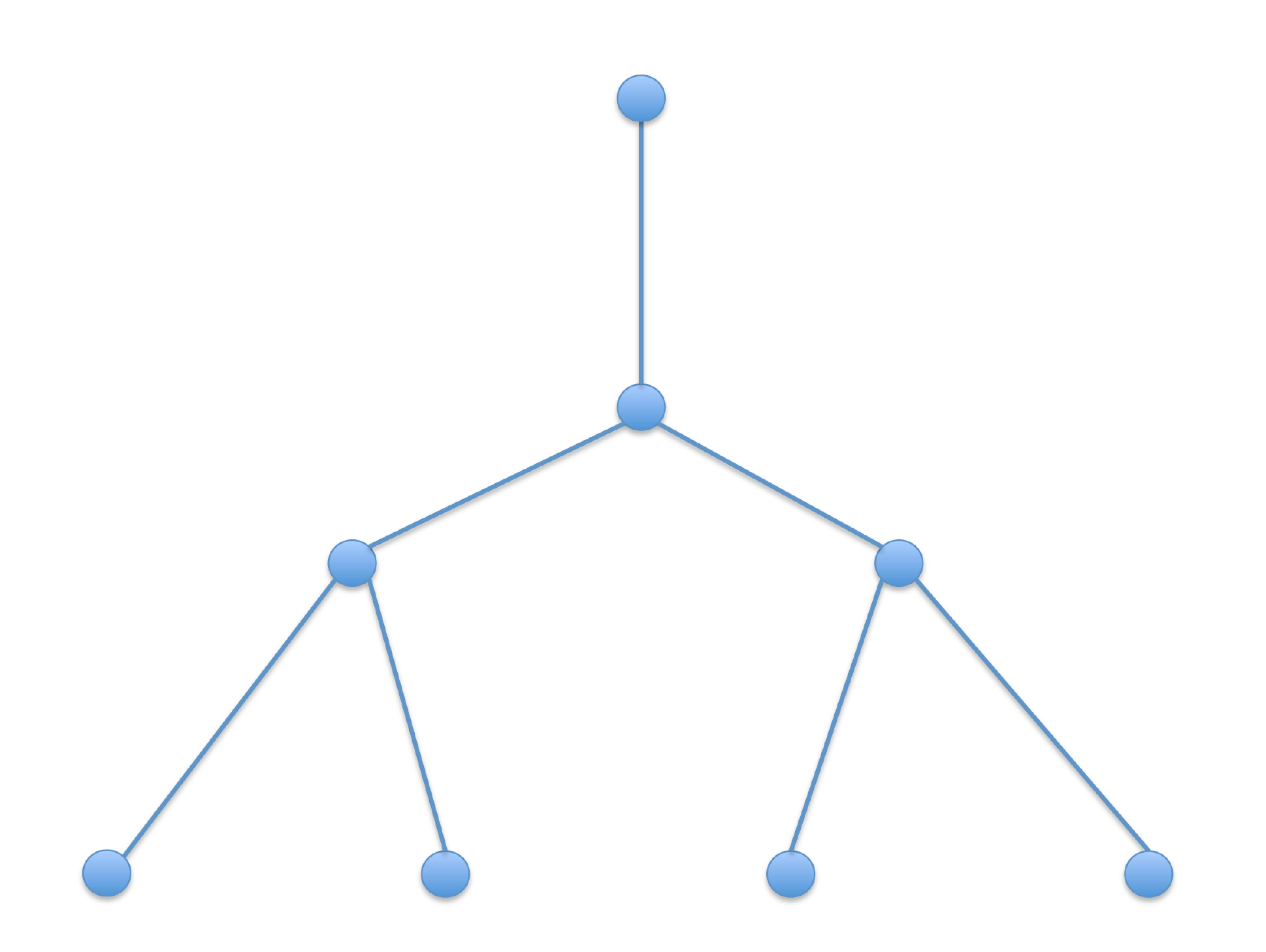} }
\subfloat[Connected leaves]{\label{fig:regular2}
\includegraphics[width=0.3\linewidth] {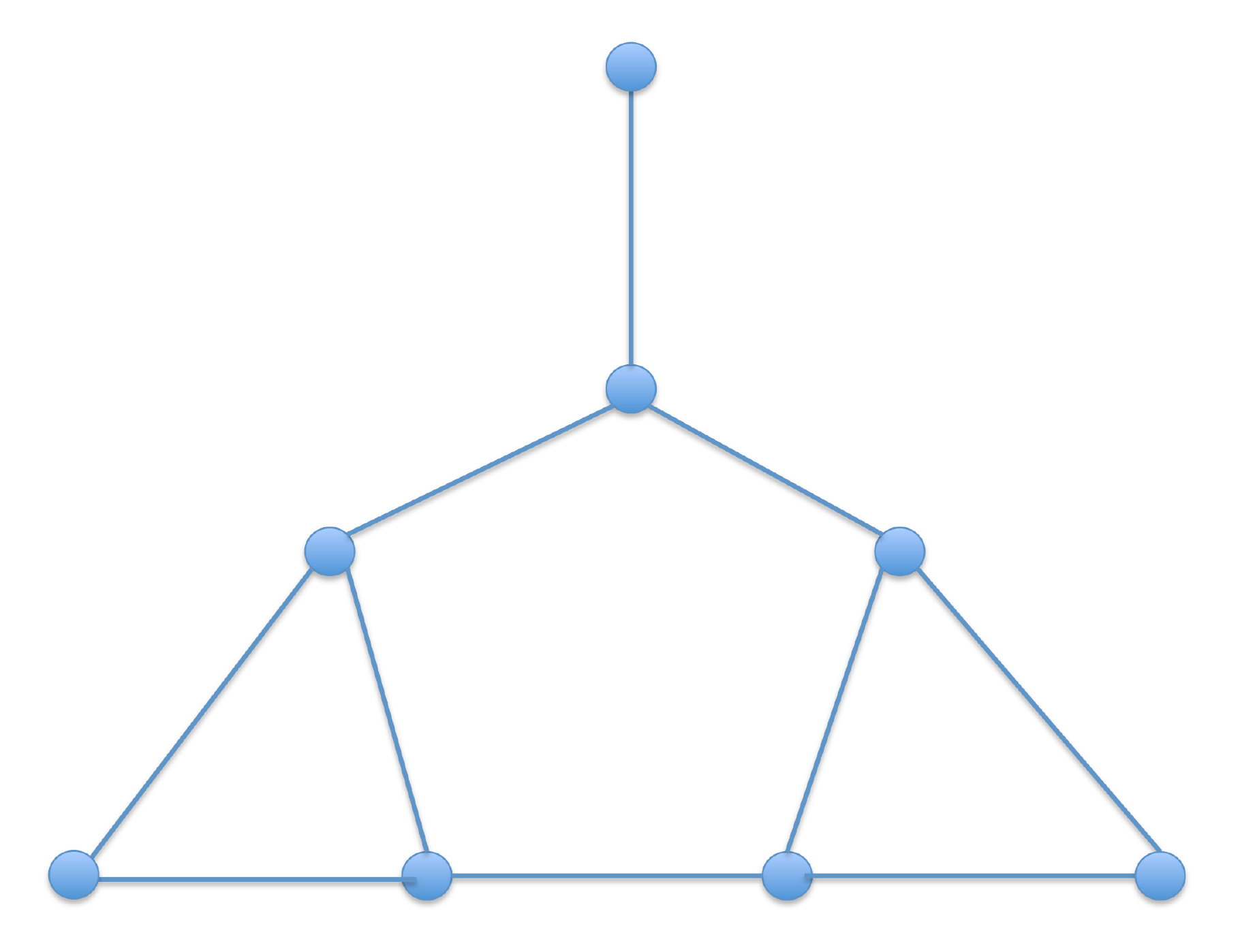}}
\subfloat[3-regular graph]{\label{fig:regular3}
\includegraphics[width=0.3\linewidth] {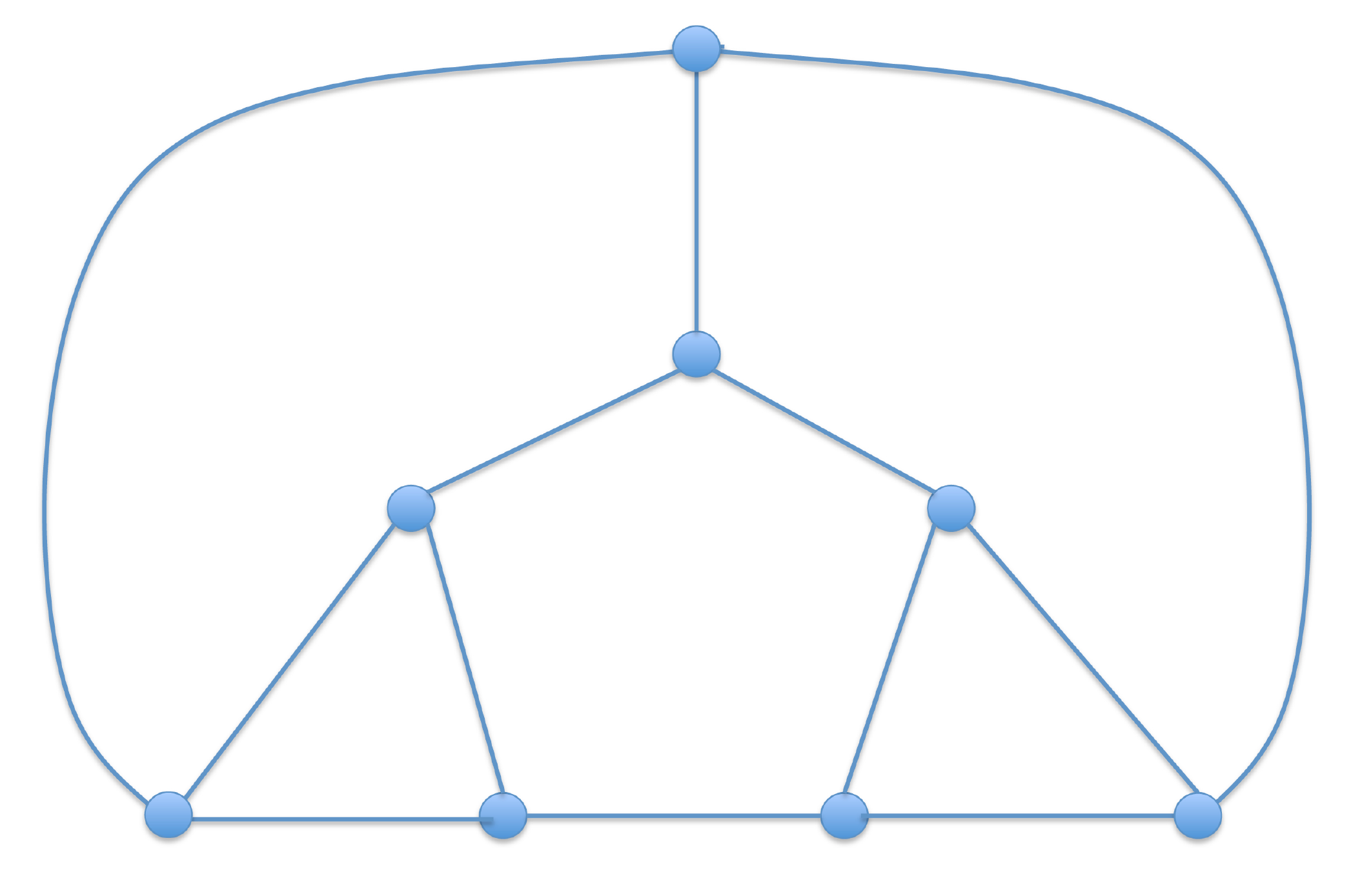}}
\caption{The construction of the 3-regular graph over $2^3=8$ nodes used in Theorem~\ref{thm:mlogm}.}
\vskip -1.5pc
\end{figure}
In the resulting regular undirected graph, we let
$A_{ij}(t)=\frac{1}{4}$ for all $j\in N_i(t)\cup\{i\}$ and for all $i$, so that
$\beta=\frac{1}{4}.$
The shortest path from the root agent to any other agent in the graph is at most $\lceil \frac{d}{2}\rceil$
(going down from the root of the tree to the nodes at the depth $\lceil \frac{d}{2}\rceil$, and going through the leaf nodes to reach those that are the depth larger than $\lceil \frac{d}{2}\rceil$).

Using the same construction, for all times $t$, we have that $\{A(t)\}$ is a sequence of doubly stochastic matrices, and therefore $\pi(t)=\frac{1}{m}\1$ for all $t$.
Thus, Assumption~\ref{assume:minimal} is satisfied, and the estimate in~\eqref{eq:dbs} reduces to
\begin{equation*}
\left\| x(t) - \bar x(0)\1\right\|^2
 \le \left(1-\frac{1}{4^3m \lceil\frac{d}{2}\rceil }\right)^{t-k}
  \left\| x(k) - \bar x(0)\1\right\|^2.
\end{equation*}
The result follows by noting that $d=\log_2{m}$.
\end{proof}
Theorem~\ref{thm:mlogm} shows that the exponential convergence rate with the ratio of the order
$1-O(\frac{1}{m\log_2 m})$ is achievable for consensus on some tree-like regular undirected graphs.
This improves the best known bound with the ratio of the order $1- O(\frac{1}{m^2})$
for undirected graphs and doubly stochastic matrices~\cite{Nedic2009b}.
We next consider the implication of Theorem~\ref{thm:key} for the convergence of matrix products
\[A(t:k)\triangleq A(t)\cdots A(k+1)A(k)\qquad\hbox{for all }t\ge k\ge0,\]
where $A(t:k)\triangleq A(k)$ whenever $t=k$.

\begin{theorem}\label{thm:matrix-conv}
If Assumption~\ref{assume:minimal},
then for all $t\ge k\ge0$,
\begin{equation*}
\left\| A(t:k)-\1\pi(k)'\right\|^2
 \le  \frac{1}{\delta}  \left(1-\frac{\delta \b^2}{4p^* }\right)^{t-k}
\left\| I- \1\pi(k)'\right\|^2. 
\end{equation*}
\end{theorem}
\begin{proof}
By Theorem~\ref{thm:key} and the fact that $\pi'(s)x(s)=\pi'(0)x(0)$ for all $s$, we have that for all $t\ge k\ge0$,
\begin{eqnarray*}
&&\sum_{i=1}^m \pi_i(t)\left(x_{i}(t) - \pi(k)'x(k) \right)^2 \cr
&& \le  \left(1-\frac{\delta \b^2}{4p^* }\right)^{t-k}
 \sum_{j=1}^m \pi_j(k)\left(x_{j}(k) - \pi(k)'x(k) \right)^2. 
\end{eqnarray*}
Since $\pi_i(k)\le 1$ for all $i$ and $k$, and $\pi_i(t)\ge\delta$ by Assumption~\ref{assume:minimal}(d),
it follows that for all $t\ge k\ge0$,
\begin{equation*}
\left\| x(t) - \pi(k)'x(k) \,\1\right\|^2
 \le  \frac{1}{\delta} \left(1-\frac{\delta \b^2}{4p^* }\right)^{t-k}
 \left\| x(k) - \pi(k)'x(k)\,\1\right\|^2. 
\end{equation*}
Noting that $x(t)=A(t:k)x(k)$ and $\pi(k)'x(k)\,\1=\1\pi(k)'\, x(k)$, we can write: for all $t\ge k\ge0,$
\begin{equation}\label{eq:trunc}
\left\| [A(t:k)-\1\pi(k)']x(k)\right\|^2
 \le  \frac{1}{\delta} \left(1-\frac{\delta \b^2}{4p^* }\right)^{t-k}
 \left\| [I- \1\pi(k)']x(k)\right\|^2. 
\end{equation}
Since the matrices $A(t)$ do not depend on the state variables $x(s),$ $0\le s<t$, the situation is similar
to constructing
$\{x(t)\}_{t\ge k}$ by the truncated matrix sequence
$\{A(t)\}_{t\ge k}$, where the dynamic is started at time $k$ in any state $x(k)$.
Then, relation~\eqref{eq:trunc} can be seen to hold for any $x(k)\in\R^n$.
Let $x(k)=x\in\R^n$ and obtain for all $t\ge k\ge0$,
\begin{equation*}
\sup_{x\ne 0}\frac{\left\| [A(t:k)-\1\pi(k)']x\right\|^2}{\|x\|^2}
 \le  \frac{1}{\delta}  \left(1-\frac{\delta \b^2}{4p^* }\right)^{t-k}
\sup_{x\ne0}\frac{\left\| [I- \1\pi(k)']x\right\|^2}{\|x\|^2}, 
\end{equation*}
which is equivalent to the stated relation.
\end{proof}

We have the following immediate consequence of Theorem~\ref{thm:matrix-conv}, by letting $t\to\infty$.

\begin{corollary}\label{cor:ergodic}
Under Assumption~\ref{assume:minimal}, the  sequence $\{A(t)\}$ is ergodic:
$\lim_{t\to\infty} A(t)\cdots A(k)=\1\pi(k)'$ for all  $k\ge0.$
\end{corollary}

\section{Constrained Consensus}\label{sec:constrained-consensus}
In this section, we consider consensus problems where the agent values are
constrained to given sets. Such constraints are inevitable in a number of applications including motion planning and alignment problems, where each agent's position is limited to a
certain region or range \cite{Pappas2008}.
Constrained consensus was first introduced in \cite{NOP2010} where
a simple discrete-time projected constrained consensus algorithm was proposed.
The analysis of the algorithm in \cite{NOP2010} relies on convergence properties of doubly stochastic
matrices. An alternative analysis developed in \cite{ren12}
gets around this limitation and also takes into account transmission delays,
but the proofs are intricate and no convergence rate results are established.
In \cite{barrier}, a continuous-time constrained consensus algorithm
was proposed using logarithmic barrier functions.
In \cite{discarded} and \cite{singapore}, discrete-time constrained consensus algorithms were presented
for a special case in which the variable of each agent is a scalar quantity.

In the sequel, we will follow the algorithm in \cite{NOP2010}.
Unlike the existing analysis in \cite{NOP2010,ren12}, we here
adopt dynamic system point of view and apply a Lyapunov approach, as done in the unconstrained consensus problem.
This approach would allow us to provide an elegant proof of convergence
and characterize the convergence rate under appropriate assumptions.

\subsection{Projected Weighted-Averaging Algorithm}

We assume that each agent has a constraint set $X_i\subseteq\R^n$, which is a convex and closed,
and the agents need to agree on a common point $c\in\cap_{i=1}^m X_i$.
We will work under the following assumption on the sets $X_i$.
\begin{assumption}\label{assume:sets}
The sets $X_i\subseteq\R^n$ are nonempty, closed, and convex, and their
intersection is nonempty, i.e.,
$X\triangleq \cap_{i=1}^m X_i\ne\emptyset.$
\end{assumption}
The constrained consensus problem is as follows.\\
{\bf [Constrained Consensus]} \
{\it Assuming that each agent $i$ knows only its set $X_i$, design a distributed
algorithm obeying the communication structure given by graph $G_t$ at each time $t$ and ensuring that,
for every set of initial values $\bx_i(0)\in\mathbb{R}^n$, $i\in [m]$, the following limiting behavior
emerges:
$\lim_{t\to\infty} \bx_i(t)=c$
for all $i\in [m]$ and some $c\in X$.}\\
To solve the constrained consensus problem, we consider the algorithm proposed in~\cite{NOP2010},
which has the following form.
Assuming that each agent starts with some initial vector $\bx_i(0)\in X_i$ at time $t=0$,
each agent $i$ updates at times $t=1,2,\ldots,$ as follows:
\begin{align}\label{eq:walgo}
\bw_i(t+1) & =  \sum_{j=1}^m A_{ij}(t) \bx_j(t),\cr
\bx_i(t+1) & =  \mathbb{P}_{X_i}[\bw_i(t+1)],
\end{align}
where $\Pr_{X_i}[\cdot]$ is the Euclidean projection on the set $X_i$.

We will show that, under Assumption~\ref{assume:minimal}
and Assumption~\ref{assume:sets}, the algorithm converges to a consensus point in
the intersection set $X$. However, unlike the results for unconstrained consensus problems,
we cannot characterize the consensus point more precisely.
We will also prove that, under some further conditions on the sets $X_i$,
the convergence rate of the algorithm is linear.
The behavior of the algorithm~\eqref{eq:walgo} is very similar to that of
the basic weighted-averaging algorithm  in~\eqref{eq:wcalgo} for the unconstrained consensus.
The intuition
comes from the following observation: the iterates of the algorithm~\eqref{eq:walgo} satisfy
$\bx_i(t+1)  =  \mathbb{P}_{X_i}\left[ \sum_{j=1}^m A_{ij}(t) \bx_j(t)\right].$
The inner averaging mapping (defined through $A(t)$) possesses some nice contraction properties under
Assumption~\ref{assume:minimal} on the graphs and the matrices $A(t)$.
This mapping is followed by a projection mapping, which is non-expansive.
Thus, one would expect that the resulting composite map 
is also contractive,
with a nearly the same contraction constant as the averaging map.

The non-expansiveness and few other properties of the projection map are summarized below.
Given a (nonempty) closed convex set $Y\subseteq\rn$,
the projection mapping $y\mapsto\Pr_Y[y]$ is non-expansive, i.e.,
\begin{equation}\label{eq:non-ex}
\|\Pr_Y[x]- y\|\le \|x-y\|\quad\hbox{for all $x\in\R^n$ and $y\in Y$},
\end{equation}
which is one of the key properties used in
the analysis of projection-based approaches.
This and other properties of the projection mapping can be found, for example,
in~\cite{Facchinei2003}, Volume 2, 12.1.13 Lemma, page 1120.
Another useful relation for the projection mapping is given by a variational inequality:
\begin{equation}\label{eq:optimalp}
\left(\Pr_Y[x]-x\right)'(y-\Pr_Y[x])\ge 0 
\end{equation}
for all $x\in\re^n$ and $y\in Y$.
The relation in~\eqref{eq:optimalp}
can be obtained by noting that the vector $\Pr_Y[x]$ is the unique solution of
the minimization problem $\min_{y\in Y}\|y-x\|^2$ and
by using the optimality condition for the solution.
The formal proof of relation~\eqref{eq:optimalp} can be found for example in~\cite{ourbook},
Proposition 2.2.1(b), page 55.

\subsection{Quadratic Lyapunov Comparison Function}\label{sec:lyapunov-constr}
Our choice of Lyapunov function is similar to the Lyapunov comparison function~\eqref{eq:vq}
for the weighted-averaging algorithm in the case of an unconstrained consensus (see Section~\ref{sec:conv}). The similarity is in the use of an adjoint sequence $\{\pi(t)\}$ associated with
the matrix sequence $\{A(t)\}$ (cf.~\eqref{eq:adjoint}); however, there is a slight difference in the choice of the centering term $\nu'x$ in~\eqref{eq:vq},
which is replaced by an arbitrary value.
Specifically, we consider the function of the following form: for all $t\ge0$ and $y\in\R^n$,
\begin{align}\label{eq:funv}
\V(t,y)\triangleq\sum_{i=1}^m\pi_i(t)
\left\|\bx_i(t) - y \right\|^2. 
\end{align}
When the values of $y$ are constrained so that $y\in X$,
the function $\V$ has an important decrease property.
To establish that property we use the following result.

\begin{lemma}\label{lemma:wlema}
Let $v\in\R^m$ be a given vector and let $\phi\in\R^m$ be a given stochastic vector.
Then, we have for any $s\in\R$,
\[(\phi'v-s)^2
= \sum_{j=1}^m\phi_j(v_j-s)^2-\frac{1}{2}\sum_{j=1}^m\sum_{\ell=1}^m\phi_j\phi_\ell(v_j-v_\ell)^2.\]
\end{lemma}
\begin{proof}
We note that $\phi'\1=1$ since $\phi$ is stochastic vector. Thus, we have
$\phi'v-s=\phi'(v-s\1)=\sum_{j=1}^m\phi_j(v_j-s).$
Therefore, by taking the square we obtain
\begin{eqnarray*}
(\phi'v-s)^2  =  \sum_{j=1}^m\sum_{\ell=1}^m \phi_j\phi_\ell(v_j-s)(v_\ell-s).
\end{eqnarray*}
Using the identity
$ab=\frac{1}{2}\left[ a^2 +b^2 - (a-b)^2\right],$
which is valid for any $a,b\in\R$,
we can further write
\begin{align*}
&(\phi'v-s)^2
=  \frac{1}{2}\sum_{j=1}^m\sum_{\ell=1}^m \phi_j\phi_\ell\left[(v_j-s)^2  +(v_\ell-s)^2-(v_j-v_\ell)^2\right]\cr
& =  \frac{1}{2}\sum_{j=1}^m\phi_j(v_j-s)^2 \left(\sum_{\ell=1}^m \phi_\ell \right)
  +\frac{1}{2}\sum_{\ell=1}^m \phi_\ell (v_\ell-s)^2\left(\sum_{j=1}^m\phi_j\right) \cr
&-\frac{1}{2}\sum_{j=1}^m\sum_{\ell=1}^m \phi_j\phi_\ell(v_j-v_\ell)^2\cr
&= \sum_{j=1}^m\phi_j(v_j-s)^2 - \frac{1}{2}\sum_{j=1}^m\sum_{\ell=1}^m \phi_j\phi_\ell(v_j-v_\ell)^2,
\end{align*}
where the last equality is obtained by using $\phi'\1=1$.
\end{proof}

Using Lemma~\ref{lemma:wlema}, we have the following decrease property for the function $\V(t,y)$
for $y\in X$.

\begin{theorem}\label{thm:fun}
Let Assumption~\ref{assume:minimal} and Assumption~\ref{assume:sets} hold.
Then, along the sequences $\{\bx_i(t)\}$, $i\in[m]$, produced by the algorithm~\eqref{eq:walgo} we have
for any initial vectors $\bx_i(0)\in X_i$, for $t\ge0$ and $y\in X$,
\begin{eqnarray*}
\V(t+1,y)
\le  \V(t,y) - \frac{\delta \b^2}{4 p^*}\,\max_{j,\ell\in V} \left\|\bx_{j}(t)  - \bx_{\ell}(t) \right\|^2,
\end{eqnarray*}
where the constants $\beta>0$ and $\delta>0$ are from
Assumptions~\ref{assume:minimal}(c) and~\ref{assume:minimal}(d), respectively, while
$p^*=\max_{t\ge0} p^*(t)$ with $p^*(t)$ being the maximum number of edges in any of the paths
from a root node to any other node in the tree $\mathsf{T}_t$ from Assumption~\ref{assume:minimal}(a).
\end{theorem}

\begin{proof}
From the definition of $\bw_i(t+1)$ in~\eqref{eq:walgo}, using the fact that the matrix $A(t)$ is stochastic
and applying Lemma~\ref{lemma:wlema} (where $\phi'=A_{i:}(t)$),
we see that the following relation is valid for each coordinate
index $\kappa\in[n]$ of the vector $\bw_i(t+1)$: for any $s\in\R$,
\begin{align*}
&([\bw_i(t+1)]_\kappa-s)^2
 = \sum_{j=1}^mA_{ij}(t)([\bx_j(t)]_\kappa - s)^2 \cr
& \  -\frac{1}{2}\sum_{j=1}^m\sum_{\ell=1}^m A_{ij}(t)A_{i\ell}\ell([\bx_j(t)]_\kappa - [\bx_\ell(t)]_\kappa)^2.
\end{align*}
Let $c\in\R^n$ be an arbitrary vector. Then,
by letting $s=c_\k$ in the preceding relation
and by summing over all coordinate indices $\kappa\in[n]$, we obtain the following relation:
for any $c\in\R^n$, for all $i\in[m]$ and all $t\ge0$,
\begin{align*}
\|\bw_i(t+1) - c\|^2
& = \sum_{j=1}^m A_{ij}(t)\|\bx_j(t) - c\|^2 \cr
& \ -\frac{1}{2}\sum_{j=1}^m\sum_{\ell=1}^m A_{ij}(t) A_{i\ell}(t)\|\bx_j(t)-\bx_\ell(t)\|^2.
\end{align*}
By multiplying with $\pi_i(t+1)$
and then summing over all $i$, we have for any $c\in\R^n$ and all $t\ge0$,
\begin{align}\label{eq:inter-rel}
&\sum_{i=1}^m\pi_i(t+1)\|\bw_i(t+1) - c\|^2 \cr
&=\sum_{i=1}^m\pi_i(t+1)\sum_{j=1}^m A_{ij}(t)\|\bx_j(t) - c\|^2 -\D(t),
\end{align}
where the decrement $\D(t)$ is given by: for all $t\ge0$,
\begin{equation}\label{eq:dec1}
\D(t)
= \frac{1}{2}\sum_{i=1}^m \pi_i(t+1)\sum_{j=1}^m\sum_{\ell=1}^m A_{ij}(t) A_{j\ell}(t)\|\bx_j(t)-\bx_\ell(t)\|^2
\end{equation}

Now, we consider the $\bx$-iterates.
By the definition of $\bx_i(t+1)$ in~\eqref{eq:walgo}, we have $\bx_i(t+1)=\Pr_{X_i}[\bw_i(t+1)]$.
Thus, by the  non-expansiveness property
of the projection map $x\mapsto \Pr_{X_i}[x]$
(see~\eqref{eq:non-ex}), we obtain for all $i$, all $t\ge0$, and all $y\in X$
(note $X\subseteq X_i$ for all $i$):
$\|\bx_i(t+1)-y\|^2\le \|\bw_i(t+1)-y\|^2. $
Therefore, by multiplying with $\pi_i(t+1)$
and then summing over all $i$, and using the definition of $\V$, we see that
\begin{eqnarray}\label{eq:relac1}
\V(t+1,y) \le\sum_{i=1}^m\pi_i(t+1)\|\bw_i(t+1)-y\|^2.
\end{eqnarray}
Letting $c=y$ in~\eqref{eq:inter-rel} and combining the resulting relation with inequality~\eqref{eq:relac1},
we obtain
\begin{eqnarray*}
\V(t+1,y) \le 
\sum_{i=1}^m\pi_i(t+1)\sum_{j=1}^m A_{ij}(t)\|\bx_j(t) - y\|^2 -\D(t).
\end{eqnarray*}
Exchanging the order of summations yields
\begin{eqnarray}\label{eq:relac2}
\V(t+1,y) & \le & \sum_{j=1}^m \left(\sum_{i=1}^m\pi_i(t+1)A_{ij}(t)\right) \|\bx_j(t) - y\|^2 - \D(t)\cr
&= & \sum_{j=1}^m \pi_j(t)\|\bx_j(t) - y\|^2  - \D(t),
\end{eqnarray}
where in the last equality we use $\pi_j(t)=\sum_{i=1}^m\pi_i(t+1)A_{ij}(t)$
(see the adjoint dynamic in~\eqref{eq:adjoint}).
Relation~\eqref{eq:relac2} and the definition of $\V(t,y)$ imply that
\begin{eqnarray}\label{eq:inter-rel1}
\V(t+1,y)
\le\V(t,y) - \D(t)
\ \hbox{for all $t\ge0$ and $y\in X$}.
\end{eqnarray}

It remains to bound the decrement $\D(t)$ in~\eqref{eq:inter-rel1} from below.
We note that the decrement $\D(t)$ defined in~\eqref{eq:dec1}
is a vector analog of the decrement $D(t)$ in Lemma~\ref{lem:Dbound}.
In particular, by defining the decrement $D_\k(t)$ for each coordinate sequence of $\bx_i(t)$,
it can be seen that
\begin{equation}\label{eq:dec-vec}
\D(t)=\sum_{\k=1}^n D_\k(t),\end{equation}
where for each coordinate $\kappa\in[n]$ and for all $t\ge0$,
\begin{equation}\label{eq:dec2}
{\small D_\k (t)=\frac{1}{2}
\sum_{i=1}^m \pi_i(t+1)\sum_{j=1}^m\sum_{\ell=1}^m A_{ij}(t)A_{i\ell}(t)
\left([x_j(t)]_\kappa-[x_\ell(t)]_\k \right)^2.}
\end{equation}
Observe that the bound of Lemma~\ref{lem:Dbound} is valid for each of the decrements $D_\k(t)$,
i.e., for all $\k\in[n]$ and $t\ge0$,
\[D_\k(t)\ge \frac{\delta \b^2}{4 p^*(t)}\,\max_{j,\ell\in [m]} \left([\bx_{j}(t)]_\k - [\bx_{\ell}(t)]_\k \right)^2.\]
By using $p^*(t)\le p^*$ and by summing the resulting inequalities over $\k\in[n]$,
from relations~\eqref{eq:dec-vec} and~\eqref{eq:dec2}
we obtain
\[\D(t)\ge \frac{\delta \b^2}{4 p^*}\,\sum_{\k=1}^n \max_{j,\ell\in [m]}
\left( [\bx_{j}(t)]_\k  - [\bx_{\ell}(t)]_\k \right)^2\quad\hbox{for }t\ge0.\]
By noting that
\[\sum_{\k=1}^n \max_{j,\ell\in [m]} \left( [\bx_{j}(t)]_\k  - [\bx_{\ell}(t)]_\k \right)^2
\ge \max_{j,\ell\in [m]} \left\|\bx_{j}(t)  - \bx_{\ell}(t) \right\|^2,\]
we arrive at the following bound
\begin{eqnarray*}
\D(t)\ge \frac{\delta \b^2}{4 p^*}\,\max_{j,\ell\in [m]} \left\|\bx_{j}(t)  - \bx_{\ell}(t) \right\|^2
\quad\hbox{for }t\ge0,
\end{eqnarray*}
which when combined with relation~\eqref{eq:inter-rel1} yields the stated relation.
\end{proof}

Theorem~\ref{thm:fun} provides the key relation that we use to establish the
convergence of the projection-based consensus algorithm, as seen in the next section.

\subsection{Convergence and Convergence Rate Results}
We first show that the algorithm correctly solves the constrained consensus problem.
Then, we investigate the rate of convergence of the algorithm in general case and some special instances.
\subsubsection{Convergence}\label{sec:conv-cc}

The following result proves that the iterates of the algorithm converge to a common point in the set $X$.
\begin{theorem}\label{thm:conv-conprob}
Let Assumption~\ref{assume:minimal} and Assumption~\ref{assume:sets} hold.
Then, the sequences $\{\bx_i(t)\}$, $i\in[m]$, produced by the algorithm~\eqref{eq:walgo}
are bounded, i.e., there is a scalar $\rho>0$ such that
\[\|\bx_i(t)\|\le \rho\qquad\hbox{for all $i\in [m]$ and all $t\ge0$},\]
and they
converge to a common point $x^*\in X$:
\[\lim_{t\to\infty} \bx_i(t)=x^*\qquad\hbox{for some $x^*\in X$ and for all $i\in [m]$.}\]
\end{theorem}

\begin{proof}
We use Theorem~\ref{thm:fun}, where we
let $\tau$ and $T$ be arbitrary times with $T>\tau\ge0$.
By summing the relations given in Theorem~\ref{thm:fun} over $t=\tau,\ldots, T-1$,
we obtain for all $y\in X$ and all $T>\tau\ge0$,
\begin{eqnarray}\label{eq:mr1}
\V(T,y)
\le \V(\tau,y)
- \frac{\delta \b^2}{4 p^*}\,\sum_{t=\tau}^{T-1}\max_{j,\ell\in [m]} \left\|\bx_{j}(t)  - \bx_{\ell}(t) \right\|^2.
\end{eqnarray}

Based on relation~\eqref{eq:mr1}, we first show that each sequence $\{\bx_i(t)\}$ is bounded.
By the definition of $\V(t,y)$, from~\eqref{eq:mr1} it follows that for all $y\in X$ and $T>\tau\ge0$,
\begin{align}\label{eq:mr2}
\sum_{i=1}^m \pi_i(T)\|\bx_i(T) & - y\|^2
\le  \sum_{j=1}^m \pi_j(\tau)\|\bx_j(\tau) - y\|^2\cr
 & - \frac{\delta \b^2}{4 p^*}\,\sum_{t=\tau}^{T-1}\max_{j,\ell\in [m]} \left\|\bx_{j}(t)  - \bx_{\ell}(t) \right\|^2.
 \quad
\end{align}
Letting $\tau=0$ and dropping the non-negative terms in~\eqref{eq:mr2},
we find that for all $y\in X$ and all $T>0$,
\[\sum_{i=1}^m \pi_i(T)\|\bx_i(T) - y\|^2
\le \sum_{j=1}^m \pi_j(0)\|\bx_j(0) - y\|^2.\]
By letting $y\in X$ be arbitrary but fixed and using the fact that the adjoint sequence $\{\pi(t)\}$
is uniformly bonded  away from zero (cf.~Assumption~\ref{assume:minimal}(d)), we conclude that
each sequence $\{\bx_i(t)\}$ is bounded, i.e., there is a scalar $\rho>0$ such that
\[\|\bx_i(t)\|\le \rho\qquad\hbox{for all $i\in [m]$ and all $t\ge0$},\]
where $\rho$ depends on $\pi(0)$, the initial points $\bx_i(0),i\in [m],$ the parameter $\delta$
and the chosen point $y\in X$.

Thus, every sequence $\{\bx_i(t)\}$ has accumulation points. We next show that
all the accumulation points of these sequences coincide, i.e.,
\begin{equation}\label{eq:same-ap}
\lim_{t\to\infty}\|\bx_i(t) - \bx_j(t)\|=0\qquad\hbox{for all $i,j\in [m]$}.
\end{equation}
This follows from~\eqref{eq:mr2}, where by letting $\tau=0$
and using non-negativity of $\V(T,y)$ we find that for all $T>0,$
\[\frac{\delta \b^2}{4 p^*}\,\sum_{t=0}^{T-1}\max_{j,\ell\in [m]} \left\|\bx_{j}(t)  - \bx_{\ell}(t) \right\|^2
\le \sum_{j=1}^m \pi_j(0)\|\bx_j(0) - y\|^2 .\]
Therefore, by letting $T\to\infty$ we conclude that
the sequences $\{\bx_j(t)\}$ have the same accumulation points (i.e.,~\eqref{eq:same-ap} is valid).
Since each sequence $\{\bx_i(t)\}$ lies in the set $X_i$ and each set $X_i$ is closed,
it follows the accumulation points of each $\{\bx_i(t)\}$ lie in the set $X_i$.
Furthermore, since the accumulation points are the same for all of the sequences $\{\bx_i(t)\}$, $i\in [m]$,
the accumulation points must be in the intersection of the sets $X_i$, i.e., in the set $X$.

Finally, we show that the sequences  $\{\bx_j(t)\}$ can have only one accumulation point, thus
showing that they converge to a common point in the set $X$.
To prove this, we argue by contraposition. Suppose that there are two accumulation points
for the sequences $\{\bx_i(t)\}$, $i\in [m]$.
Let $\{t_s\}$ and $\{\tau_s\}$ be the time sequences along which the iterates $\{\bx_i(t)\}$ converge, respectively,
to two distinct points, say $\check{x}\in X$ and $\hat x\in X$, with $\check{x}\ne\hat x$,
\begin{equation}\label{eq:mr3}
\lim_{s\to\infty} \bx_i(t_s) =\check{x},\quad \lim_{s\to\infty} \bx_i(\tau_s)=\hat x,
\qquad\hbox{for all }i\in V.
\end{equation}
Without loss of generality let us assume that $t_s>\tau_s$ for all $s\ge 1$ (for otherwise we can construct such subsequences from $\{t_s\}$ and $\{\tau_s\}$).
In relation~\eqref{eq:mr2}, we let $T=t_s$ and $\tau=\tau_s$ for any $s\ge 1$, and thus, obtain
(by omitting the non-negative terms) for all $y\in X$,
\[\sum_{i=1}^m \pi_i(t_s)\|\bx_i(t_s) - y\|^2
\le \sum_{j=1}^m \pi_j(\tau_s)\|\bx_j(\tau_s) - y\|^2 \ \hbox{for all }s\ge1.\]
Letting $y=\hat x$ and recalling that the adjoint sequence $\{\pi(t)\}$ is bounded away from 0,
we see that
\[\delta \sum_{i=1}^m \|\bx_i(t_s) - \hat x\|^2
\le \sum_{j=1}^m \pi_j(\tau_s)\|\bx_j(\tau_s) - \hat x\|^2 \qquad\hbox{for all }s\ge1.\]
Now, letting $s\to\infty$  we have
\begin{align*}
\delta \lim_{s\to\infty} \left( \sum_{i=1}^m \|\bx_i(t_s) - \hat x\|^2 \right)
& \le  \lim_{s\to\infty} \left(\sum_{j=1}^m \pi_j(\tau_s)\|\bx_j(\tau_s) - \hat x\|^2 \right)\cr
& \le  \sum_{j=1}^m \lim_{s\to\infty} \|\bx_j(\tau_s) - \hat x\|^2 ,\end{align*}
where in the last inequality we use $0\le \pi_j(t)\le 1$ for all $j$ and $t$.
From relation~\eqref{eq:mr3} it follows that
\[\delta \sum_{i=1}^m \|\check{x} - \hat x\|^2 \le 0,\]
thus implying  $\check{x}=\hat x$, which is a contradiction.
Hence, the sequences $\{\bx_i(t)\}$, $i\in [m]$, must be convergent.
\end{proof}

Theorem~\ref{thm:conv-conprob} shows that Proposition 2 in~\cite{NOP2010} holds under weaker
assumptions on the graphs and the weights.
At first, the requirement in~\cite{NOP2010} that each matrix $A(t)$ is doubly stochastic is relaxed.
At second, while here we assume that each of the graphs $G_t$ is rooted,
the results easily extend to the case studied in~\cite{NOP2010} by assuming that
the graphs are rooted over at most $B$ units of time and
that the absolute probability sequence exists for such unions of the graphs.

\subsubsection{Convergence Rate}\label{sec:conv-cc}
Our convergence rate results are obtained for sets $X_i$ that satisfy a certain regularity condition
which relates the distances from a given point to the sets $X_\ell$
with the distance from the point to the intersection set $X=\cap_{i=1}^m X_i$.
One relation that among these distances always holds. In particular, since $X\subseteq X_i$ for all $i$,
it follows that
\begin{equation}\label{eq:crit-bel}
\dist(x,X_i)\le \dist(x,X) \ \hbox{for all $x\in\rn$ and $i\in [m]$}.\end{equation}
In our analysis, we need
an upper bound on $\dist(x,X)$ in terms of the distances $\dist(x,X_i)$, $i\in [m]$.
A related generic question is:
when the distances of a given point $y$ to a collection of closed convex sets $\{Y_i,i\in\cal{I}\}$
can be related to the distance of $y$ from the intersection set $Y=\cap_{i\in \cal I} Y_i\ne\emptyset$?
This question has been studied in the
optimization literature within the terminology of {\it error bounds} or {\it metric regularity}.
In this literature, loosely speaking, the question is when the distance $\dist(y,Y)$ is bounded from
above by a constant factor of the maximum distance $\max_{i\in\cal{I}}\dist(y,Y_i)$.
In general, the index set $\cal{I}$ can be infinite, but we restrict our attention to finite index sets only.

We will use the following definition of set regularity.
\begin{definition}\label{def:reg-set}
Let $Z\subseteq\R^n$ be a nonempty set.
We say that a (finite) collection of closed convex sets $\{Y_i,i\in\cal{I}\}$ is regular (in Euclidian norm)
with respect to the set $Z$,
if there is a constant $r\ge1$ such that
\[\dist(y,Y)\le r\, \max_{i\in\cal{I}} \left\{\dist(y,Y_i)\right\}\qquad\hbox{for all $y\in Z$}.\]
We refer to the scalar $r$ as a {\it regularity constant}.
When the preceding relation holds with $Z=\rn$, we say that the sets $\{Y_i,i\in\cal{I}\}$
are {\it uniformly regular}.
\end{definition}
In view of relation~\eqref{eq:crit-bel} it follows that the regularity constant $r$ must satisfy
$r\ge1.$
Note that the regularity constant $r$ in Definition~\ref{def:reg-set} depends on the set $Z$.
It also depends on the choice of the metric and the geometry of the sets $\{Y_i,i\in\cal{I}\}$.
In general, it is hard to compute $r$, but our algorithm does not require the knowledge of such a constant.
We just provide a convergence rate result that captures the dependence on $r$.

In view of Theorem~\ref{thm:conv-conprob}, the iterate sequences $\{\bx_i(t)\}$, $i\in[m]$,
are contained a ball $B(0,\rho)$  centered at the origin with a radius $\rho$.
We will assume that the sets $X_\ell$ are regular with respect to the ball $B(0,\rho)$.
Later in Section~\ref{sec:set-reg} we discuss some sufficient conditions for this regularity assumption to hold.
Under such a regularity assumption,
we show a result that is critical in the subsequent convergence rate analysis.
\begin{lemma}\label{lem:crit}
Let Assumption~\ref{assume:sets} hold.
Assume further that the sets $\{X_i,i\in [m]\}$ are regular with respect to a set $Z\subseteq\R^n$
with a regularity constant~$r\ge1$, and assume that
$\left(X_1\times\cdots\times X_m\right)
\cap \left(Z\times\cdots\times Z\right)\ne\emptyset.$
Let $\phi\in\R^m$ be a given stochastic vector.
Then, for all
$(\bx_1,\ldots,\bx_m)\in \left(X_1\times\cdots\times X_m\right)
\cap \left(Z\times\cdots\times Z\right)$ we have
\[\max_{j,\ell\in [m]}\|\bx_j - \bx_\ell\|
\ge \frac{1}{r+1}\,\max_{p\in [m]} \left\|\bx_p-\Pr_X\left[ \sum_{i=1}^m \phi_i \bx_i\right]\right\|.\]
\end{lemma}

\begin{proof}
Let $(\bx_1,\ldots,\bx_m)\in \left(X_1\times\cdots\times X_m\right)
\cap \left(Z\times\cdots\times Z\right)$ be arbitrary,
and define $u=\sum_{i=1}^m \phi_i \bx_i$.
Let $\ell\in [m]$ be arbitrary. Consider estimating $\|\bx_\ell -\Pr_X[u]\|$ as follows:
\begin{eqnarray*}
\|\bx_\ell - \Pr_X[u]\|
&\le& \|\bx_\ell - \Pr_X[\bx_\ell]\| + \|\Pr_X[\bx_\ell] - \Pr_X[u]\|\cr
&\le & r\max_{j\in [m]}\left\{ \dist(\bx_\ell,X_j)\right\} + \|\bx_\ell - u\|.
\end{eqnarray*}
where the first inequality uses the triangle inequality for the norm. The second inequality
uses the fact $\|\bx_\ell - \Pr_X[\bx_\ell]\|=\dist(\bx_\ell,X)$ and the set regularity assumption
for the first term (i.e., $\dist(y,X)\le r\max_{i}\dist(y,X_i)$ for all $y\in Z$ and the fact $\bx_\ell \in Z$),
while the second term is estimated by using the non-expansiveness property of the projection map
(see~\eqref{eq:non-ex}).
By the definition of the projection, we have
\[\dist(\bx_\ell,X_j)=\min_{y\in X_j}\|\bx_\ell - y\|\le \|\bx_\ell - \bx_j\|,\]
where the inequality follows by $\bx_j\in X_j$ for all $j$.
Thus,
\begin{equation}\label{eq:crit1}
\|\bx_\ell -\Pr_X[u]\| \le  r\max_{j\in [m]}\|\bx_\ell - \bx_j\| + \|\bx_\ell - u\|.
\end{equation}
Consider now the term $\|\bx_\ell-u\|$. By the definition of $u$, this vector is a convex combination of
points $\bx_i,i\in [m],$ since $\phi$ is a stochastic vector.
Thus, by the convexity of the Euclidean norm,  it follows that
\[\|\bx_\ell -u\|= \left\|\sum_{i=1}^m \phi_i(\bx_\ell-x_i) \right\|\le \sum_{i=1}^m \phi_i\|\bx_\ell-\bx_i\|
\le \max_{i\in [m]}\|\bx_\ell - \bx_i\|.\]
By substituting the preceding estimate in relation~\eqref{eq:crit1}, we obtain
\[\|\bx_\ell-\Pr_X[u]\| \le  (r+1)\max_{j\in [m]}\|\bx_\ell - \bx_j\|.\]
So far the index  $\ell$ was arbitrary, so by taking the maximum over all $\ell\in [m]$,
we find that
\[\max_{\ell\in [m]} \|\bx_\ell - \Pr_X[u]\| \le  (r+1)\max_{j,\ell\in [m]}\|\bx_\ell - \bx_j\|,\]
and the desired relation follows after dividing by $r+1$.
\end{proof}

With Lemma \ref{lem:crit} in place, we investigate the rate of decrease of the Lyapunov comparison function
$\V(t,y)$, as
given in~\eqref{eq:funv}. We have the following result.

\begin{theorem}\label{thm:rate2}
Let Assumption~\ref{assume:minimal} and Assumption~\ref{assume:sets} hold.
Assume further that the sets $\{X_i,i\in [m]\}$ are regular,
with a regularity constant~$r\ge1$, with respect to a ball $B(0,\rho)$ which contains all the iterates
$\{\bx_i(t)\}$ generated by the algorithm~\eqref{eq:walgo}.
Consider the following vectors
\begin{equation}\label{eq:def-uv}
\bu(t)=\sum_{i=1}^m \pi_i(t) \bx_i(t), \  \bv(t)=\Pr_X[u(t)], \ \hbox{for all }t\ge0.
\end{equation}
Then, the Lyapunov comparison function $\V(t,\bv(t))$ decreases at a geometric rate:
for all $t\ge0$,
\begin{eqnarray*}
\V\left(t+1,\bv(t+1)\right)
\le \left(1- \frac{\delta \b^2}{4 p^* (r+1)^2}\right) \,\V(t,\bv(t)),
\end{eqnarray*}
where the scalars $\delta,\beta\in (0,1)$ and the integer $p^*\ge 1$
are the same as in Theorem~\ref{thm:fun}.
\end{theorem}
\begin{proof}
In Theorem~\ref{thm:fun} we let $y=\bv(t)$ with $\bv(t)\in X$ and we use the
definition of $\bu(t)$.
Then, we have for all $t\ge0$,
\begin{align}\label{eq:rater1}
\V\left(t+1,\bv(t)\right)
\le
\V\left(t,\bv(t) \right)) - \frac{\delta \b^2}{4 p^*}\,\max_{j,\ell\in [m]} \left\|\bx_{j}(t)  - \bx_{\ell}(t) \right\|^2.
\end{align}
Next, we consider the term $V\left(t+1,\bv(t)\right)$. We have
\begin{eqnarray*}
&& \V\left(t+1,\bv(t)\right)
= \sum_{i=1}^m\pi_i(t+1)\|\bx_i(t+1)-\bv(t)\|^2\cr
&= &\sum_{i=1}^m\pi_i(t+1)\left\|\bx_i(t+1)-\bv(t+1) + \left(\bv(t+1) -\bv(t)\right) \right\|^2.
\end{eqnarray*}
By expanding the squared-norm terms, we obtain
\begin{eqnarray*}
&&\V\left(t+1,\bv(t)\right)
\ge \sum_{i=1}^m\pi_i(t+1)\|\bx_i(t+1)-\bv(t+1) \|^2 \cr
&& + 2\left(\sum_{i=1}^m\pi_i(t+1)\bx_i(t+1)-\bv(t+1) \right)'\left(\bv(t+1) -\bv(t)\right),
\end{eqnarray*}
where the inequality is obtained by dropping the term
$\|\bv(t+1) -\bv(t)\|^2$. 
In view of the definition of the vector $\bu(t+1)$ (cf.~\eqref{eq:def-uv}), it follows that
\begin{align*}
\V\left(t+1,\bv(t)\right)
 & =   \sum_{i=1}^m\pi_i(t+1)\|\bx_i(t+1)-\bv(t+1) \|^2 \cr
& \ +2\left(\bu(t+1)-\bv(t+1) \right)'\left(\bv(t+1) -\bv(t)\right),
\end{align*}
Since $\bv(t+1)$ is the projection of $\bu(t+1)$ on the set $X$
and since $\bv(t)\in X$, it further follows that
\[\left(\bu(t+1)-\bv(t+1) \right)'\left(\bv(t+1) -\bv(t)\right)\ge0\]
(see relation~\eqref{eq:optimalp}).
Hence
\begin{align*}
\V\left(t+1,\bv(t)\right)
& \ge  \sum_{i=1}^m\pi_i(t+1)\|\bx_i(t+1)-\bv(t+1) \|^2 \cr
&=\V(t+1,\bv(t+1)).
\end{align*}
By combining the preceding relation with~\eqref{eq:rater1}
we can conclude that  for all $t\ge0$,
\begin{eqnarray}\label{eq:rater2}
\V\left(t+1,\bv(t+1)\right)
\le
\V\left(t,\bv(t) \right)) - \frac{\delta \b^2}{4 p^*}\,\max_{j,\ell\in [m]} \left\|\bx_{j}(t)  - \bx_{\ell}(t) \right\|^2.
\end{eqnarray}

To estimate the term $\max_{j,\ell\in [m]} \left\|\bx_{j}(t)  - \bx_{\ell}(t) \right\|^2$ from below
we use Lemma~\ref{lem:crit} with the following identification: $Z=B(0,\rho)$,
$\bx_i=\bx_i(t)$, $\phi=\pi(t)$ and $u=\bu(t)$, and we note that $\bx_i(t)\in Z$ for all $i$ and $t$.
Thus, by Lemma~\ref{lem:crit} we have
\[\max_{j,\ell\in [m]}\|\bx_j(t) - \bx_\ell(t)\|\ge \frac{1}{r+1}\,\max_{p\in [m]}\|\bx_p(t)-\Pr_X[\bu(t)]\|.\]
In our notation, we have $\bv(t)=\Pr_X[\bu(t)]$ (see~\eqref{eq:def-uv}), so by using
$\bv(t)$ and by taking squares
in the preceding relation we obtain
\[\max_{j,\ell\in [m]}\|\bx_j(t) - \bx_\ell(t)\|^2
\ge \frac{1}{(r+1)^2}\,\max_{p\in [m]}\|\bx_p(t) - \bv(t)\|^2.\]
Since the vector $\pi(t)$ is stochastic, we have
\[\max_{p\in [m]}\|\bx_p(t) - \bv(t)\|^2\ge \sum_{i=1}^m \pi_i(t) \|\bx_i(t) - \bv(t)\|^2 =\V(t,\bv(t)),\]
where the equality uses the definition of $\V(t,y)=\sum_{i=1}^m \pi_i(t) \|\bx_i(t) - y\|^2$
(see~\eqref{eq:funv}).
Therefore
\begin{equation}\label{eq:rater3}
\max_{j,\ell\in [m]}\|\bx_j(t) - \bx_\ell(t)\|^2\ge \frac{1}{(r+1)^2}\V(t,\bv(t)).
\end{equation}
By substituting the estimate~\eqref{eq:rater3} into inequality~\eqref{eq:rater2} we obtain
the desired relation.
\end{proof}

Using the decrease rate result for the Lyapunov comparison function $\V(t,y)$ of
Theorem~\ref{thm:rate2}, and the properties of the adjoint dynamics,
we can now estimate the rate of convergence of the iterates $\{\bx_i(t)\}$.

\begin{theorem}\label{thm:rate3}
Let Assumption~\ref{assume:minimal} and Assumption~\ref{assume:sets} hold.
Assume further that the sets $\{X_i,i\in [m]\}$ are regular, with a regularity constant~$r\ge1$,
with respect to a ball $B(0,\rho)$ which contains all the iterates
$\{\bx_i(t)\}$ generated by the algorithm~\eqref{eq:walgo}.
Then, the sequences $\{\bx_i(t)\}$, $i\in[m],$ are such that for all $t\ge0,$
\[\sum_{j=1}^m \dist^2\left(\bx_j(t), X\right)
\le \frac{1}{\delta} \left(1- \frac{\delta \b^2}{4 p^* (r+1)^2}\right)^t \,\V(0,\bv(0)),\]
where $\bv(0)=\Pr_X[\bu(0)]$ with $\bu(0)=\sum_{j=1}^m\pi_j(0) \bx_j(0)$, while the scalars
$\delta,\beta\in (0,1)$ and the integer $p^*\ge 1$ are the same as in Theorem~\ref{thm:fun}.
\end{theorem}

\begin{proof}
From Theorem~\ref{thm:rate2} it can be seen that
$\V\left(t,\bv(t)\right)
\le \left(1- \frac{\delta \b^2}{4 p^* (r+1)^2}\right)^t \,\V(0,\bv(0))$
for all $t\ge0.$
The result follows by recalling that
$\V(t,y)=\sum_{i=1}^m \pi_i(t)\|\bx_i(t) - y\|^2$, recalling the definition of $\bv(t)$ (see~\eqref{eq:def-uv}),
and using the fact that the vectors $\pi(t)$ have uniformly bounded entries from below by $\delta>0$
(cf.~Assumption~\ref{assume:minimal}(d)).
\end{proof}

Theorem~\ref{thm:rate3} extends the convergence rate result obtained originally in~\cite{NOP2010},
where the convergence rate was analyzed for a special case when the matrices $A(t)$ are doubly stochastic, and
the graph is static and complete, i.e., $A(t)=\frac{1}{m}\1\1'$ for all $t$.

\subsubsection{Sufficient Conditions for Set Regularity}\label{sec:set-reg}
We discuss two cases of sufficient conditions for the set regularity property, namely,
the case of a polyhedral set $X$, and the case of $X$ with a nonempty interior.\\

\noindent{\bf Polyhedral Set $X$}.\quad
Let $X\subseteq\R^n$ be a nonempty polyhedral set.
We will show that use the description of $X$ in terms of
linear inequalities,
\[X=\{x\in\rn \mid a_i'x\le b_i, \ i\in\cal I\},\]
where $\cal I$ is a finite index set, $a_i\in \rn$ and $b_i\in\re$ for all $i$.
For such a set, Hoffman in~\cite{Hoffman1952} had shown that the distance from
any point $x\in\rn$ to the set $X$ is bounded from above by the maximal distance from $x$ to any of the
hyperplanes defined by the linear inequalities, i.e.,
that there exists a constant
$r\ge1$ such that
\begin{equation}\label{eq:h-bound}
\dist(x,X)\le r\max_{i\in{\cal I} }\left\{ \dist (x,H_i)\right\} \qquad\hbox{for all }x\in\rn,
\end{equation}
where, for every $i$, the set $H_i$ is the hyperplane given by
$H_i=\{x\in\rn\mid a_i'x \le b_i\}$, while the constant $r$ depends
on the set of normals $\{a_i, i\in\cal I\}$ that define the hyperplanes $\{H_i, i\in\cal I\}$.
We will refer to this relation as the {\it  Hoffman bound}.
We will use this bound to show that,
when each set $X_i$ is polyhedral, the sets $X_i$ are uniformly regular.

\begin{proposition}\label{prop:lin-sets}
       Assume that each set $X_j$, $j\in [m]$, is given by
      $X_j=\{x\in\rn\mid (a_\ell^{(j)})'x\le b_{\ell}^{(j)},\, \ell\in {\cal I}_j\}.$
      Also, assume that $X=\cap_{i=1}^m X_i$ is nonempty. Then, the sets $X_i$ are uniformly regular
       with the regularity constant equal to the constant $r$
       in the Hoffman bound~\eqref{eq:h-bound}, where ${\cal I}=\cup_{j=1}^m {\cal I}_j$,
       i.e.,
       \[\dist(x,X)\le r\max_{i\in [m]} \left\{ \dist (x,X_i) \right\} \qquad\hbox{for all }x\in\rn.\]
            \end{proposition}
     \begin{proof}
       Note that the set $X$ is the intersection of the hyperplanes that define the sets $X_i$, i.e.,
       $X=\cap_{j=1}^m\left(\cap_{\ell\in {\cal I}_j} H_{\ell}^{(j)}\right),$
       where $H_{\ell}^{(j)} = \{x\mid (a_\ell^{(j)})'x\le b_{\ell}^{(j)}\}$.
      By the Hoffman bound, there is an $r\ge1$ such that
      \begin{equation}\label{eq:dist1}
      \dist(x,X)\le r\max_{j\in [m]} \max_{\ell\in {\cal I}_j}
      \left\{ \dist (x,H_{\ell}^{(j)} \right\}\qquad\hbox{for all }x\in\rn.
      \end{equation}
      For every $j\in [m]$, we have
      $H_{\ell}^{(j)}\supseteq X_j$ for all $\ell\in{\cal I}_j$,
      thus implying that for every $j\in [m]$,
      \[\max_{\ell\in {\cal I}_j} \left\{ \dist(x,H_{\ell}^{(j)}) \right\}
             \le \dist (x,X_i)\qquad \hbox{for all }x\in\rn.\]
       The preceding relation and~\eqref{eq:dist1} yield
       \[\dist(x,X)\le r\max_{j\in [m]} \left\{ \dist (x,X_j)\right\}\qquad\hbox{for all }x\in\rn.\]
       Thus, the sets $X_i,i\in [m]$ are uniformly regular.
     \end{proof}

    Hence, when the sets $X_i$ are polyhedral, they are uniformly regular and thus,
    also regular with respect to any ball $B(0,\rho)$ that contains the sequences $\{\bx_i(t)\}$. Consequently,
    when the sets $X_i$ are polyhedral, the regularity condition
    of Theorem~\ref{thm:rate3} holds.\\

    \noindent
   {\bf Set $X$ with Nonempty Interior}.\quad
    The regularity condition also holds
    when the interior of the intersection set $X$ is nonempty. The proof
    uses some ideas from~\cite{Gubin1967} (see the proof of Lemma 5 there).
    However, in this case, the set regularity
    property is not global.

       \begin{proposition}\label{prop:lin-rate2}
       Let Assumption~\ref{assume:sets} hold, and assume that
       the set $X=\cap_{j\in [m]} X_j$ has a nonempty interior,
       i.e., there is
       a vector $\bar x\in X$ and a scalar $\theta>0$ such that
       $\{z\in\re^n\mid \|z-\bar x\|\le \theta\}\subseteq X.$
       Let $Y\subseteq\rn$ be a bounded set.
       Then, we have
       \[\dist(x,X)\le r \max_{j\in [m]} \left\{ \dist (x,X_j)\right\}\qquad\hbox{for all }x\in Y,\]
       with $r=\frac{1}{\theta}\max_{y\in Y}\|y-\bar x\|$.
     \end{proposition}

    \begin{proof}
       Let $x\in \re^n$ be arbitrary. Define
       $\e=\max_{j\in [m]} \left\{\dist^2(x,X_j)\right\}$ and
       consider the vector
       $y=\frac{\e}{\e+\theta}\,\bar x + \frac{\theta}{\e +\theta}\, x.$
       We show that $y\in X$. To see this note that we can write for each~$j\in [m]$,
       \[y=\frac{\e}{\e+\theta}\, \left(\bar x + \frac{\theta}{\e}\,(x-\Pr_{X_j}[x]) \right)
           +\frac{\theta}{\e +\theta}\,\Pr_{X_j}[x].\]
           The vector $z=\bar x + \frac{\theta}{\e}\,(x-\Pr_{X_j}[x]) $ satisfies
       \[\|z-\bar x\|=\frac{\theta}{\e}\,\|x-\Pr_{X_j}[x]\|
       \le\frac{\theta}{\e}\,\max_{j\in [m]} \|x-\Pr_{X_j}[x]\|  = \theta,\]
       where the last equality follows by the definition of $\e$ and
       $\dist(x,X_j)=\|x-\Pr_{X_j}[x]\|$.
       Thus, since $\bar x$ is an interior point of $X$, it follows that
       $z\in X\subseteq X_i$ for all  $i\in [m].$
       Since the vector $y$ is a convex combination of $z\in X_j$ and
       $\Pr_{X_j}[x]\in X_j$, by the convexity of the set $X_j$, it follows that $y\in X_j$.

       Therefore, for each $j$, the vector $y$ can be written as a convex
       combination of two points in $X_j$, implying that $y\in X_j$ for all $j\in [m]$.
       Consequently, we have $y\in X$, so that
       $\dist(x,X)\le \|x-y\|=\frac{\e}{\e +\theta}\, \|x-\bar x\|\le \frac{\e}{\theta}\, \|x-\bar x\|.$
       Using the definition of $\e$, we obtain
       $\dist(x,X)\le
          \frac{1}{\theta}\|x-\bar x\|\,\max_{j\in [m]}\left\{\dist(x,X_j) \right\},$
         which is valid for any $x\in\rn$.
       By using $\|x-\bar x\|\le \max_{x\in Y}\|x-\bar x\|$, we arrive at
       \[\dist(x,X)\le
          \left(\frac{1}{\theta}\max_{y\in Y}\|y-\bar x\|\right)
          \max_{j\in [m]}\left\{\dist(x,X_j) \right\} \hbox{for all }x\in Y.\]
  \end{proof}

\section{Conclusion}\label{sec:concl}
We have investigated the properties of the weighted-averaging dynamic for consensus problem using
Lyapunov approach. We have established new convergence rate results in terms of the longest shortest path
of spanning trees contained in the graph. For constrained consensus, we established exponential convergence rate
assuming some regularity conditions on the constraint sets. These results easily extend to the cases where the underlying graphs are not necessarily rooted at every instant, but rather rooted over a period of time.

 \section*{Acknowledment}
 The authors are deeply grateful to A.S.\ Morse, A.\ Olshevsky and B.\ Touri for valuable and insightful discussions that have significantly influenced this work.

\bibliography{distributed,soomin001,broadcast,directed-opt,ji}

\begin{thebibliography}{10}
\providecommand{\url}[1]{#1}
\csname url@samestyle\endcsname
\providecommand{\newblock}{\relax}
\providecommand{\bibinfo}[2]{#2}
\providecommand{\BIBentrySTDinterwordspacing}{\spaceskip=0pt\relax}
\providecommand{\BIBentryALTinterwordstretchfactor}{4}
\providecommand{\BIBentryALTinterwordspacing}{\spaceskip=\fontdimen2\font plus
\BIBentryALTinterwordstretchfactor\fontdimen3\font minus
  \fontdimen4\font\relax}
\providecommand{\BIBforeignlanguage}[2]{{%
\expandafter\ifx\csname l@#1\endcsname\relax
\typeout{** WARNING: IEEEtran.bst: No hyphenation pattern has been}%
\typeout{** loaded for the language `#1'. Using the pattern for}%
\typeout{** the default language instead.}%
\else
\language=\csname l@#1\endcsname
\fi
#2}}
\providecommand{\BIBdecl}{\relax}
\BIBdecl

\bibitem{Reynolds1987}
C.~W. Reynolds, ``Flocks, herds, and schools: a distributed behavioral model,''
  in \emph{Proceedings of the 14th Annual Conference on Computer Graphics and
  Interactive Techniques}, 1987, pp. 25--34.

\bibitem{Boyd2006}
S.~Boyd, A.~Ghosh, B.~Prabhakar, and D.~Shah, ``Randomized gossip algorithms,''
  \emph{IEEE Transactions on Information Theory}, vol.~52, no.~6, pp.
  2508--2530, 2006.

\bibitem{Bullo2004}
J.~Cort\'es, S.~Mart\'inez, T.~Karata\c{s}, and F.~Bullo, ``Coverage control
  for mobile sensing networks,'' \emph{IEEE Transactions on Robotics and
  Automation}, vol.~20, no.~2, pp. 243--255, 2004.

\bibitem{Morse2007}
J.~Lin, A.~S. Morse, and B.~D.~O. Anderson, ``The multi-agent rendezvous
  problem. {P}art 1: the synchronous case,'' \emph{SIAM Journal on Control and
  Optimization}, vol.~46, no.~6, pp. 2096--2119, 2007.

\bibitem{Evans2004}
L.~Hu and D.~Evans, ``Localization for mobile sensor networks,'' in
  \emph{Proceedings of the 10th annual international conference on Mobile
  computing and networking}, 2004, pp. 45--57.

\bibitem{Francis2009}
L.~Krick, M.~E. Broucke, and B.~A. Francis, ``Stabilisation of infinitesimally
  rigid formations of multi-robot networks,'' \emph{International Journal of
  Control}, vol.~82, no.~3, pp. 423--439, 2009.

\bibitem{Morse2012}
J.~Liu, N.~Hassanpour, S.~Tatikonda, and A.~S. Morse, ``Dynamic threshold
  models of collective action in social networks,'' in \emph{Proceedings of the
  51st IEEE Conference on Decision and Control}, 2012, pp. 3991--3996.

\bibitem{Bullo2013}
F.~D{\"o}rfler, M.~Chertkov, and F.~Bullo, ``Synchronization in complex
  oscillator networks and smart grids,'' \emph{Proceedings of the National
  Academy of Sciences}, vol. 110, no.~6, pp. 2005--2010, 2013.

\bibitem{TsThesis}
J.~N. Tsitsiklis, ``Problems in {D}ecentralized {D}ecision {M}aking and
  {C}omputation,'' Ph.D. dissertation, Department of Electrical Engineering and
  Computer Science, MIT, 1984.

\bibitem{Ts1986}
J.~N. Tsitsiklis, D.~P. Bertsekas, and M.~Athans, ``Distributed asynchronous
  deterministic and stochastic gradient optimization algorithms,'' \emph{IEEE
  Transactions on Automatic Control}, vol.~31, no.~9, pp. 803--812, 1986.

\bibitem{Morse2003}
A.~Jadbabaie, J.~Lin, and A.~S. Morse, ``Coordination of groups of mobile
  autonomous agents using nearest neighbor rules,'' \emph{IEEE Transactions on
  Automatic Control}, vol.~48, no.~6, pp. 988--1001, 2003.

\bibitem{Murray2004}
R.~Olfati-Saber and R.~M. Murray, ``Consensus problems in networks of agents
  with switching topology and time-delays,'' \emph{IEEE Transactions on
  Automatic Control}, vol.~49, no.~9, pp. 1520--1533, 2004.

\bibitem{Moreau2005}
L.~Moreau, ``Stability of multiagent systems with time-dependent communication
  links,'' \emph{IEEE Transactions on Automatic Control}, vol.~50, no.~2, pp.
  169--182, 2005.

\bibitem{Ren2005}
W.~Ren and R.~Beard, ``Consensus seeking in multiagent systems under
  dynamically changing interaction topologies,'' \emph{IEEE Transactions on
  Automatic Control}, vol.~50, no.~5, pp. 655--661, 2005.

\bibitem{Basar2007}
A.~Kashyap, T.~Ba\c{s}ar, and R.~Srikant, ``Quantized consensus,''
  \emph{Automatica}, vol.~43, no.~7, pp. 1192--1203, 2007.

\bibitem{Blondel2005}
V.~D. Blondel, J.~M. Hendrickx, A.~Olshevsky, and J.~N. Tsitsiklis,
  ``Convergence in multiagent coordination, consensus, and flocking,'' in
  \emph{Proceedings of the 44th IEEE Conference on Decision and Control}, 2005,
  pp. 2996--3000.

\bibitem{Oh2007}
S.~Oh, L.~Schenato, P.~Chen, and S.~Sastry, ``Tracking and coordination of
  multiple agentsusing sensor networks: {S}ystem {D}esign, {A}lgorithms and
  {E}xperiments,'' \emph{Proceedings of the IEEE}, vol.~95, no.~1, 2007.

\bibitem{Bullo2009}
F.~Bullo, J.~Cort\'es, and S.~Mart\'{\i}nez, \emph{Distributed Control of
  Robotic Networks}.\hskip 1em plus 0.5em minus 0.4em\relax Applied Mathematics
  Series. Princeton University Press, 2009.

\bibitem{Mesbahi2010}
M.~Mesbahi and M.~Egerstedt, \emph{Graph Theoretic Methods for Multiagent
  Networks}.\hskip 1em plus 0.5em minus 0.4em\relax Princeton, NJ, USA:
  Princeton University Press, 2010.

\bibitem{Martinoli2013}
A.~Martinoli, F.~Mondada, G.~Mermoud, N.~Correll, M.~Egerstedt, A.~Hsieh,
  L.~Parker, and K.~Stoy, \emph{Distributed Autonomous Robotic Systems}.\hskip
  1em plus 0.5em minus 0.4em\relax Springer Tracts in Advanced Robotics,
  Springer-Verlag, 2013.

\bibitem{Lopes2008}
C.~Lopes and A.~Sayed, ``Diffusion least-mean squares over adaptive networks:
  Formulation and performance analysis,'' \emph{IEEE Trans. Signal Process.},
  vol.~56, no.~7, pp. 3122--3136, 2008.

\bibitem{Sayed2012}
A.~Sayed, ``Diffusion adaptation over networks,'' 2012, to appear in
  E-Reference Signal Processing, R. Chellapa and S. Theodoridis, editors,
  Elsevier, 2013. Also available online as arXiv:1205.4220v1, 2012.

\bibitem{RamThesis}
S.~Ram, ``Distributed optimization in multi-agent systems: Applications to
  distributed regression,'' Ph.D. dissertation, University of Illinois at
  Urbana-Champaign, 2009.

\bibitem{alexthesis}
A.~Olshevsky, ``Efficient information aggregation for distributed control and
  signal processing,'' Ph.D. dissertation, MIT, 2010.

\bibitem{kunalthesis}
K.~Srivastava, ``Distributed optimization with applications to sensor networks
  and machine learning,'' Ph.D. dissertation, University of Illinois at
  Urbana-Champaign, Industrial and Enterp. Systems Eng., 2011.

\bibitem{NOP2010}
A.~Nedi\'c, A.~Ozdaglar, and P.~A. Parrilo, ``Constrained consensus and
  optimization in multi-agent networks,'' \emph{IEEE Transactions on Automatic
  Control}, vol.~55, no.~4, pp. 922--938, 2010.

\bibitem{touri2014}
B.~Touri and A.~Nedi\'c, ``Product of random stochastic matrices,'' \emph{IEEE
  Transactions on Automatic Control}, vol.~59, no.~2, pp. 437--448, 2014.

\bibitem{touribook}
B.~Touri, \emph{Product of random stochastic matrices and distributed
  averaging}.\hskip 1em plus 0.5em minus 0.4em\relax Springer-Verlag, Berlin,
  2012.

\bibitem{tourithesis}
------, ``Product of random stochastic matrices and distributed averaging,''
  Ph.D. dissertation, University of Illinois at Urbana-Champaign, Industrial
  and Enterp. Systems Eng., 2011.

\bibitem{Boyd2005}
L.~Xiao, S.~Boyd, and S.~Lall, ``A scheme for robust distributed sensor fusion
  based on average consensus,'' in \emph{Proceedings of the 4th International
  Conference on Information Processing in Sensor Networks}, 2005, pp. 63--70.

\bibitem{Boyd2004}
L.~Xiao and S.~Boyd, ``Fast linear iterations for distributed averaging,''
  \emph{Systems and Control Letters}, vol.~53, no.~1, pp. 65--78, 2004.

\bibitem{Morse2008a}
M.~Cao, A.~S. Morse, and B.~D.~O. Anderson, ``Reaching a consensus in a
  dynamically changing environment: a graphical approach,'' \emph{SIAM Journal
  on Control and Optimization}, vol.~47, no.~2, pp. 575--600, 2008.

\bibitem{Morse2008b}
------, ``Reaching a consensus in a dynamically changing environment:
  convergence rates, measurement delays and asynchronous events,'' \emph{SIAM
  Journal on Control and Optimization}, vol.~47, no.~2, pp. 601--623, 2008.

\bibitem{Nedic2009a}
A.~Nedi\'c and A.~Ozdaglar, ``Distributed subgradient methods for multi-agent
  optimization,'' \emph{IEEE Transactions on Automatic Control}, vol.~54,
  no.~1, pp. 48--61, 2009.

\bibitem{Morse2011}
J.~Liu, A.~S. Morse, B.~D.~O. Anderson, and C.~Yu, ``Contractions for consensus
  processes,'' in \emph{Proceedings of the 50th IEEE Conference on Decision and
  Control}, 2011, pp. 1974--1979.

\bibitem{Nedic2009b}
A.~Nedi\'c, A.~Olshevsky, A.~Ozdaglar, and J.~N. Tsitsiklis, ``On distributed
  averaging algorithms and quantization effects,'' \emph{IEEE Transactions on
  Automatic Control}, vol.~54, no.~11, pp. 2506--2517, 2009.

\bibitem{Ts2013}
A.~Olshevsky and J.~N. Tsitsiklis, ``Degree fluctuations and the convergence
  time of consensus algorithms,'' \emph{IEEE Transactions on Automatic
  Control}, vol.~58, no.~10, pp. 2626--2631, 2013.

\bibitem{Fagnani2008}
F.~Fagnani and S.~Zampieri, ``Randomized consensus algorithms over large scale
  networks,'' \emph{IEEE Journal on Selected Areas in Communications}, vol.~26,
  no.~4, pp. 634--649, 2008.

\bibitem{Bajovic2013}
D.~Bajovi\'c, J.~Xavier, J.~Moura, and B.~Sinopoli, ``Consensus and products of
  random stochastic matrices: Exact rate for convergence in probability,''
  \emph{IEEE Transactions on Signal Processing}, vol.~61, no.~10, 2013.

\bibitem{Ts2009}
A.~Olshevsky and J.~N. Tsitsiklis, ``Convergence speed in distributed consensus
  and averaging,'' \emph{SIAM Journal on Control and Optimization}, vol.~48,
  no.~1, pp. 33--55, 2009.

\bibitem{Morse2011p}
J.~Liu, S.~Mou, A.~S. Morse, B.~D.~O. Anderson, and C.~Yu, ``Deterministic
  gossiping,'' \emph{Proceedings of the IEEE}, vol.~99, no.~9, pp. 1505--1524,
  2011.

\bibitem{touri2011}
B.~Touri and A.~Nedi\'c, ``On existence of a quadratic comparison function for
  random weighted averaging dynamics and its implications,'' in
  \emph{Proceedings of the 50th IEEE Conference on Decision and Control}, 2011,
  pp. 3806--3811.

\bibitem{Kolmogorov}
A.~Kolmogoroff, ``Zur theorie der markoffschen ketten,'' \emph{Mathematische
  Annalen}, vol. 112, no.~1, pp. 155--160, 1936.

\bibitem{Blackwell1945}
D.~Blackwell, ``Finite non-homogeneous chains,'' \emph{Annals of Mathematics},
  vol.~46, no.~4, pp. 594--599, 1945.

\bibitem{julien2013}
J.~Hendrickx and J.~Tsitsiklis, ``Convergence of type-symmetric and
  cut-balanced consensus seeking systems,'' \emph{IEEE Transactions on
  Automatic Control}, vol.~58, no.~1, pp. 214--218, 2013.

\bibitem{bolouki2012}
S.~Bolouki and R.~Malham\'e, ``Theorems about ergodicity and class-ergodicity
  of chains with applications in known consensus models,'' in \emph{Proceedings
  of the 50th Annual Allerton Conference on Communication, Control, and
  Computing}, 2012, pp. 1425--1431.

\bibitem{Pappas2008}
M.~M. Zavlanos and G.~J. Pappas, ``Dynamic assignment in distributed motion
  planning with local coordination,'' \emph{IEEE Transactions on Robotics},
  vol.~24, no.~1, pp. 232--242, 2008.

\bibitem{ren12}
P.~Lin and W.~Ren, ``Distributed constrained consensus in the presence of
  unbalanced switching graphs and communication delays,'' in \emph{Proceedings
  of the 51st IEEE Conference on Decision and Control}, 2012, pp. 2238--2243.

\bibitem{barrier}
U.~Lee and M.~Mesbahi, ``Constrained consensus via logarithmic barrier
  functions,'' in \emph{Proceedings of the 50th IEEE Conference on Decision and
  Control}, 2011, pp. 3608--3613.

\bibitem{discarded}
Z.~Liu and Z.~Chen, ``Discarded consensus of network of agents with state
  constraint,'' \emph{IEEE Transactions on Automatic Control}, vol.~57, no.~11,
  pp. 2869--2874, 2012.

\bibitem{singapore}
C.~Sun, C.~J. Ong, and J.~K. White, ``Consensus control of multi-agent system
  with constraint - the scalar case,'' in \emph{Proceedings of the 52nd IEEE
  Conference on Decision and Control}, 2013, pp. 7345--7350.

\bibitem{Facchinei2003}
F.~Facchinei and J.-S. Pang, \emph{Finite-Dimensional Variational Inequalities
  and Complementarity Problems}.\hskip 1em plus 0.5em minus 0.4em\relax
  Springer-Verlag New York, 2003, vol. I and II.

\bibitem{ourbook}
D.~Bertsekas, A.~Nedi\'c, and A.~Ozdaglar, \emph{Convex Analysis and
  Optimization}.\hskip 1em plus 0.5em minus 0.4em\relax Belmont, Massachusetts:
  Athena Scientific, 2003.

\bibitem{Hoffman1952}
A.~Hoffman, ``On approximate solutions of systems of linear inequalities,''
  \emph{Journal of Research of the National Bureau of Standards}, vol.~49,
  no.~4, pp. 263--265, 1952.

\bibitem{Gubin1967}
L.~Gubin, B.~Polyak, and E.~Raik, ``The method of projections for finding the
  common point of convex sets,'' \emph{USSR Computational Mathematics and
  Mathematical Physics}, vol.~7, no.~6, pp. 1 -- 24, 1967.

\end{thebibliography}

\end{document}